\documentclass{article}
\usepackage[cm]{fullpage}

\usepackage{amsmath, amssymb, amsthm, float}
\usepackage{url}
\usepackage{tikz, here}
\usepackage{booktabs}
\newcommand{\abs}[1]{\lvert#1\rvert}

\newcommand\precdot{\mathrel{\ooalign{$<$\cr
\hidewidth\hbox{$\cdot\mkern1mu$}\cr}}}

\newcommand{\Cone}{\operatorname{Cone}}

\newcommand{\Conv}{\operatorname{Conv}}

\newcommand{\join}{\operatorname{Join}}
\newcommand{\Edges}{\operatorname{Edges}}

\newcommand{\DDef}{\operatorname{Def}}
\newcommand{\Vol}{\operatorname{Vol}}
\newcommand{\Cl}{\operatorname{Cl}}
\newcommand{\Pic}{\operatorname{Pic}}

\newcommand{\rank}{\operatorname{rank}}

\newcommand{\bC}{\ensuremath{\mathbb{C}}}
\newcommand{\bO}{\ensuremath{\mathbb{O}}}
\newcommand{\bP}{\ensuremath{\mathbb{P}}}
\newcommand{\bQ}{\ensuremath{\mathbb{Q}}}
\newcommand{\bR}{\ensuremath{\mathbb{R}}}
\newcommand{\bZ}{\ensuremath{\mathbb{Z}}}

\newcommand{\scD}{\ensuremath{\mathcal{D}}}

\newcommand{\scL}{\ensuremath{\mathcal{L}}}
\newcommand{\scO}{\ensuremath{\mathcal{O}}}
\newcommand{\scP}{\ensuremath{\mathcal{P}}}

\theoremstyle{plain}
\newtheorem{theorem}{Theorem}[section]

\newtheorem{proposition}[theorem]{Proposition}

\theoremstyle{definition}
\newtheorem{definition}[theorem]{Definition}
\newtheorem{remark}[theorem]{Remark}
\newtheorem{example}[theorem]{Example}
\theoremstyle{definition}

\title{Complete intersection Calabi--Yau threefolds in Hibi toric varieties\\ 
and their smoothing}
\author{Makoto Miura}
\date{}

\begin{document}
\maketitle
\begin{abstract}
 In this article,
 we summarize combinatorial description of
 complete intersection Calabi--Yau threefolds 
 in Hibi toric varieties.
 Such Calabi--Yau threefolds 
 have at worst conifold singularities, and
 are often smoothable to 
 non-singular Calabi--Yau threefolds.
 We focus on 
 such non-singular Calabi--Yau threefolds of Picard number one,
 and illustrate the calculation of topological invariants,
 using new motivating examples.
\end{abstract}


\section{Introduction}
%
A {\it Hibi toric variety} is defined as a projective toric variety $\bP_{\Delta(P)}$
associated with an order polytope
\begin{align}
 \Delta(P) =\left\{ (x_u)_{u\in P} \mid 0 \le x_u \le x_v \le 1
 \text{ for }u \prec v \in P \right\} \subset \bR^P,
\end{align}
for a finite poset $P=(P,\prec)$.
For example, all products of projective spaces are Hibi toric varieties; hence
at least 2590 topologically distinct non-singular Calabi--Yau threefolds are
obtained as complete intersections \cite{MR1086756}.
In general, complete intersection Calabi--Yau threefolds in Hibi toric varieties
have finite number of nodes,
and are often smoothable to non-singular Calabi--Yau threefolds by flat deformations.
Complete intersections in Grassmannians 
(or more generally in minuscule Schubert varieties)
give basic examples of such smoothing
\cite{MR1619529, MR3688804}.

The purpose of this article is to
provide a brief summary on 
combinatorial descriptions of 
complete intersection Calabi--Yau threefolds
in Hibi toric varieties and their smoothing.
Based on \cite{MR2801412},
we describe the smoothability in terms of posets
(Proposition~\ref{prop:smoothing}), and survey the calculation of 
topological invariants for resulting non-singular simply-connected
Calabi--Yau threefolds (Subsection~\ref{subsec:top}),
by focusing on the case of Picard number one for simplicity.
In addition to the summary,
we show the simply-connectedness
as a corollary of the result on small resolutions for Hibi toric varieties
(Proposition~\ref{prop:small}).
To illustrate the calculation, 
we introduce several new examples of
such non-singular Calabi--Yau threefolds
of Picard number one (Subsection~\ref{ss:new}, Table~\ref{tab:CY3}).

A Calabi--Yau threefold is a complex projective threefold $X$
with at worst canonical singularities satisfying  
$\omega_X \simeq \scO_X$ and $H^1(X, \scO_X)=0$.
There are a huge number of such threefolds, 
even non-singular.
Mirror symmetry is a conjectural duality
between a non-singular Calabi--Yau threefold $X$ and
another non-singular Calabi--Yau threefold $X^*$, 
called a {\it mirror manifold\/} for $X$.
Various non-trivial relations between $X$ and $X^*$
are expected.
For example, Hodge numbers satisfy
\begin{align}
 \label{eq:hodge}
 h^{i,j}(X) = h^{3-j, i}(X^*) \quad \text{for all $i$ and $j$}.
\end{align}

One of the big mysteries of mirror symmetry
is whether
every non-singular Calabi--Yau threefold $X$ has a mirror manifold $X^*$ or not.
Note an obvious exception in the case with $h^{2,1}(X)=0$, and that
the mirror manifold $X^*$ is not unique in general,
even as topological manifolds.
There is an excellent class of non-singular Calabi--Yau threefolds such that
the above question has an affirmative answer;
for crepant resolutions of complete intersection Calabi--Yau threefolds
in Gorenstein toric Fano varieties,
we have mirror manifolds in the same class, called the {\it
Batyrev--Borisov mirrors\/} \cite{MR1269718, 1993alg.geom.10001B}.
In order to expand this class, 
the conjectural mirror construction via conifold transitions seems to be a 
promising direction.

Let $X_0$ be a Calabi--Yau threefold
with finitely many nodes.
Suppose that 
$X_0$ admits a smoothing $X \leadsto X_0$ by a flat deformation,
and a small resolution $Y \rightarrow X_0$.
The composite operation connecting two non-singular Calabi--Yau threefolds
$X$ and $Y$ is called a \textit{conifold transition}:
\begin{align}
 \label{eq:coni}
 X \leadsto X_0 \leftarrow Y.
\end{align}
There is a natural closed immersion of the Kuranishi space $\DDef(Y)$ to $\DDef(X_0)$
\cite[Proposition 2.3]{MR1923221}, and hence, it makes sense to put them together 
into some giant moduli space.
There is a question, commonly referred to as (a version of) {\it Reid's fantasy\/}, 
which asks whether all simply-connected non-singular Calabi--Yau threefolds 
fit together into a single irreducible family via conifold transitions
\cite{MR909231}.
Suppose that
$X$ and $Y$ have torsion-free homology
for a conifold transition (\ref{eq:coni}).
{\it Morrison's conjecture\/} in \cite{MR1673108} says that the mirror manifolds
are also connected via a conifold transition of the opposite direction: 
\begin{align}
 \label{eq:conimirror}
 Y^* \leadsto Y^*_0 \leftarrow X^*.
\end{align}
Together with the spirit of Reid's fantasy, 
one may expect a mirror construction
for a large number of non-singular Calabi--Yau threefolds
from the Batyrev--Borisov mirror pairs.

We still do not know the existence of a mirror manifold $X^*$,
even for the smoothing $X\leadsto X_0$ of 
a complete intersection Calabi--Yau threefold $X_0$ in a Hibi toric variety.
Nevertheless, we can discuss the mirror symmetry 
by calculating periods and Picard--Fuchs operators for the conjectural mirror family,
as we see in Remark~\ref{rem:pf} for example.


\section{Hibi toric varieties}
\subsection{Examples}
Let us begin with simple examples of Hibi toric varieties.
For the empty poset, 
we set the Hibi toric variety $\bP_{\Delta(\varnothing)}$ to be a point.
For a singleton
$u:=\left\{ u \right\}$ (by abuse of notation),
the order polytope is a line segment $\Delta(u) = [0,1]$, and hence,
the Hibi toric variety $\bP_{\Delta(u)}$ is
a projective line $\bP^1$.

Let $P$ be a finite poset consisting of $n:=\abs{P}$ elements.
If $P$ is a {\it chain}, i.e., a totally ordered set,
the order polytope $\Delta(P)$ is a regular simplex 
defined by the inequalities
$
0 \le x_1 \le \dots \le x_n \le 1,
$
so that the Hibi toric variety $\bP_{\Delta(P)}$ is
a projective space $\bP^n$.
It is equally clear the case that $P$ is an {\it anti-chain}, 
i.e., the poset in which every pair of elements is incomparable.
In this case, the order polytope $\Delta(P)$ is a unit hypercube $[0,1]^{n}$,
and the Hibi toric variety $\bP_{\Delta(P)}$ is 
the product of $n$ copies of $\bP^1$.
\begin{example}
 \label{ex:v}
A first non-trivial example is a poset $P=\left\{ u,v,w \right\}$
with the partial order defined by $u \succ w$ and $v \succ w$.
The defining inequalities of
the order polytope $\Delta(P)$ is shown in the left of Figure~\ref{fig:op},
also depicted symbolically in the middle.
It becomes a pyramid in $\bR^P \simeq \bR^3$ 
as shown in the right of Figure~\ref{fig:op}.
Therefore, the associated Hibi toric variety $\bP_{\Delta(P)}$ is 
a projective cone over $\bP^1 \times \bP^1$ with a general apex in $\bP^3$.
\begin{figure}[H]
 \centering
 \begin{tabular}{ccc}
  \includegraphics{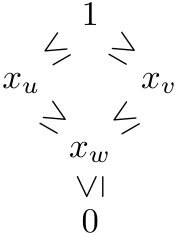}
  &
  \includegraphics{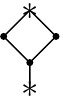}
  & 
  \includegraphics{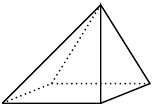}  
 \end{tabular}
 \caption{An example of order polytopes}
 \label{fig:op}
\end{figure}

\end{example}

A {\it disjoint union} $P = P_1+ P_2$ of finite posets $P_1$ and $P_2$
is a disjoint union as sets 
equipped with the partial order $\prec$ satisfying
(i) $u\in P_1$, $v\in P_1$ and $u \prec v \in P_1$ imply $u\prec v \in P$,
(ii) $u\in P_2$, $v\in P_2$ and $u \prec v \in P_2$ imply $u\prec v \in P$,
and 
(iii) $u \in P_1$ and  $v\in P_2$ imply $u \not \sim v \in P$ (i.e., $u$ and $v$ are incomparable
in $P$). 
The corresponding Hibi toric variety is
projectively equivalent to
the product of two Hibi toric varieties,
\begin{align}
 \bP_{\Delta(P_1+P_2)} \simeq
 \bP_{\Delta(P_1)} \times \bP_{\Delta(P_2)}.
\end{align}
A {\it ordinal sum} $P = P_1 \oplus P_2$ of $P_1$ and $P_2$
is a disjoint union as sets 
equipped with the partial order $\prec$ satisfying
the same (i) and (ii) as the disjoint union $P_1+P_2$ above,
and (iii)$'$ $u \in P_1$ and  $v\in P_2$ imply $u \prec v \in P$.
Note that the operation $\oplus$ is not commutative though it is associative.
The corresponding Hibi toric variety is
a special hyperplane section
of a projective join 
of two Hibi toric varieties with general positions,
\begin{align}
 \bP_{\Delta(P_1\oplus u \oplus P_2)} \simeq
 \join \left(\bP_{\Delta(P_1)}, \bP_{\Delta(P_2)} \right).
\end{align}
These operations generalize the examples,
a chain $P=\bigoplus_{i=1}^n u_i$, 
an anti-chain $P=\sum_{i=1}^n u_i$,
and $P=w \oplus (u+v) = \varnothing \oplus w \oplus (u+v)$ in Example~\ref{ex:v}. 

The posets built up 
by disjoint unions and ordinal sums
from singletons are sometimes called series-parallel posets.
One of the simplest examples that are not series-parallel is the poset 
with the {\it Hasse diagram}:
\begin{align}
 \label{eq:N}
 \centering
 \begin{tikzpicture}[scale=0.4, baseline={(0,0.15)}]
 \draw[black] (0,0) -- (0,1) -- (1,0) -- (1,1);
 \filldraw[black] (0,0) circle (2pt);
 \filldraw[black] (0,1) circle (2pt);
 \filldraw[black] (1,0) circle (2pt);
 \filldraw[black] (1,1) circle (2pt);
\end{tikzpicture}
\end{align}
Recall that, in a Hasse diagram for a poset $P$, 
a vertex represents an element of $P$
and an oriented edge represents a {\it covering relation} 
$u \precdot v$ on $P$, that is,
\begin{align}
 u \prec v \text{ and there is no } w\in P \text{ such that } u \prec w \prec v.
\end{align}
For example,
the Hasse diagram (\ref{eq:N}) represents the poset $P=\left\{ a,b,c,d \right\}$ with
$a \prec b \succ c \prec d$.
The associated Hibi toric variety $\bP_{\Delta(P)}$
is a limit of a toric degeneration of a general linear section fourfold
in a Grassmannian $G(2,5)$.

\subsection{Invariant subvarieties and singularities}
Invariant subvarieties of Hibi toric varieties are again
(projectively equivalent to) lower dimensional Hibi toric varieties.
We follow the description of invariant subvarieties by Wagner \cite{MR1382045}.

Let $P$ be a finite poset. The associated
bounded poset is defined as
\begin{align}
 \hat{P} := \hat 0 \oplus P \oplus \hat 1,
\end{align}
where $\hat 0$ and $\hat 1$ are singletons.
By definition, the elements $\hat 0$ and $\hat 1$ are
the unique minimal and the maximal elements in $\hat{P}$,
respectively.
Note that the Hasse diagram of $\hat P$ 
can be regarded as the graph describing the defining inequalities of
order polytope $\Delta(P)$, as we see in the middle of Figure~\ref{fig:op}.
We use this identification between inequalities with edges,
and variables with vertices for the Hasse diagram of $\hat P$.
Furthermore, by abuse of notation, 
we write the same symbol $P$ as the Hasse diagram of $P$.
For example, we say
that $P$ is {\it connected} if the Hasse diagram of $P$ is connected,
and $P$ is a {\it cycle} if the Hasse diagram of $P$ is a cycle 
as an unoriented graph, and so on.
\begin{definition}
 Let $\hat{P}$ be a bounded poset.
 A surjective order-preserving map
 \begin{align}
 \varphi\colon \hat{P} \rightarrow \hat{P}' 
 \end{align}
 with $\varphi(\hat 0) =\hat 0$ and
 $\varphi(\hat 1) =\hat 1$
 is called a \textit{contraction} of $\hat{P}$ if
 every fiber is connected and 
 there exists a covering relation $u \precdot v \in \hat{P}$
 for all $\bar u \precdot \bar v \in \hat{P}'$
 such that $\bar u=\varphi(u)$ and $\bar v=\varphi(v)$.
\end{definition}
There is a one-to-one correspondence between 
faces of order polytope $\Delta(P)$ and contractions of
the associated bounded poset $\hat{P}$.
More precisely, a face $\theta_\varphi$ corresponding to a contraction 
$\varphi \colon \hat{P} \rightarrow \hat{P'}$
is unimodular equivalent to the order polytope $\Delta(P')$.
In other words,
the associated invariant subvariety
of a Hibi toric variety $\bP_{\Delta(P)}$ is
projectively equivalent to the Hibi toric variety $\bP_{\Delta(P')}$,
as mentioned at the beginning of this subsection.
In particular, we have one-to-one correspondences between
facets of $\Delta(P)$ and edges of $\hat{P}$,
and vertices of $\Delta(P)$ and {\it order ideals} of $P$.
Here an order ideal is defined as a subset $\tau \subset P$ satisfying
\begin{align}
 u\in \tau, v\in P \text{ and } u \succ v \text{ imply } v \in \tau.
\end{align}
Let us write the set of edges of $\hat{P}$ as $E = \Edges(\hat{P})$,
and the set of order ideals of $P$ as $J(P)$.

We illustrate the correspondences
by using the poset $P$ in Example~\ref{ex:v}.
We have five facets corresponding to $E$, and
five vertices corresponding to $J(P)$.
The defining equalities of a face $\theta_\varphi$
can be obtained 
by making all variables in a fiber $\varphi^{-1}(\bar u)$ equal.
\begin{figure}[H]
 \centering
\begin{minipage}{20mm}
\ \begin{tikzpicture}[scale=1]
 \fill[gray!40!white] (1,1) -- (0,0) -- (1,0);
 \draw[black] (0,0) -- (1,0) -- (1,1) --(3/2,1/5)--(1,0);
 \draw[black] (0,0) -- (1,1);
 \draw[black, densely dotted] (0,0) -- (1/2,1/5) -- (3/2,1/5);
 \draw[black, densely dotted] (1/2,1/5) -- (1,1);
\end{tikzpicture}\\
$
\begin{tikzpicture}[scale=0.2, baseline={(0,0.15)}]
 \draw[black, thick] (1,3) -- (0,2);
 \draw[black] (0,2) -- (1,1) -- (1,0);
 \draw[black] (1,3) -- (2,2) -- (1,1);
 \filldraw[black] (2,2) circle (4pt);
 \filldraw[black] (1,1) circle (4pt);
 \filldraw[black] (0,2) circle (6pt);
 \draw (1,0-0.3) node[circle] {*};
 \draw[black] (1,3-0.3) node[circle] {\large{*}};
 \draw[rotate around={45:(0.5,2.5)}] (0.5,2.5) ellipse (1.5 and 0.7);
\end{tikzpicture}
\rightarrow \!\!
\begin{tikzpicture}[scale=0.2, baseline={(0,0.15)}]
 \draw[black] (1,3) -- (1,2) -- (1,1) -- (1,0);
 \filldraw[black] (1,2) circle (4pt);
 \filldraw[black] (1,1) circle (4pt);
 \draw (1,0-0.3) node[circle] {*};
 \draw[black] (1,3-0.3) node[circle] {\large{*}};
 \draw (1,3) circle (0.6);
\end{tikzpicture}
$\\
\ \begin{tikzpicture}[scale=1]
 \draw[black] (0,0) -- (1,0) -- (1,1)--(3/2,1/5)--(1,0);
 \draw[black] (0,0) -- (1,1);
 \draw[black, densely dotted] (0,0) -- (1/2,1/5) -- (3/2,1/5);
 \draw[black, densely dotted] (1/2,1/5) -- (1,1);
 \filldraw[black] (1/2,1/5) circle (1.2pt);
\end{tikzpicture}\\
$
\begin{tikzpicture}[scale=0.2, baseline={(0,0.15)}]
 \draw[gray] (1,3) -- (0,2);
 \draw[black, thick] (0,2) -- (1,1) -- (1,0);
 \draw[gray] (1,3) -- (2,2);
 \draw[black, thick] (2,2) -- (1,1);
 \filldraw[black] (2,2) circle (6pt);
 \filldraw[black] (1,1) circle (6pt);
 \filldraw[black] (0,2) circle (6pt);
 \draw[black] (1,0-0.3) node[circle] {\large{*}};
 \draw[gray] (1,3-0.3) node[circle] {*};
\end{tikzpicture}
\rightarrow \!\!
\begin{tikzpicture}[scale=0.2, baseline={(0,0.02)}]
 \draw[gray] (1,1) -- (1,0);
 \draw (1,0-0.3) node[circle] {\large{*}};
 \draw[gray] (1,1-0.3) node[circle] {*};
\end{tikzpicture}
$
\end{minipage}
\begin{minipage}{20mm}
\ \begin{tikzpicture}[scale=1]
 \fill[gray!40!white] (1,1) -- (3/2,1/5) -- (1,0);
 \draw[black] (0,0) -- (1,0) -- (1,1) --(3/2,1/5)--(1,0);
 \draw[black] (0,0) -- (1,1);
 \draw[black, densely dotted] (0,0) -- (1/2,1/5) -- (3/2,1/5);
 \draw[black, densely dotted] (1/2,1/5) -- (1,1);
\end{tikzpicture}\\
$
\begin{tikzpicture}[scale=0.2, baseline={(0,0.15)}]
 \draw[black] (1,3) -- (0,2) -- (1,1) -- (1,0);
 \draw[black] (2,2) -- (1,1);
 \draw[black, thick] (1,3) -- (2,2);
 \filldraw[black] (2,2) circle (6pt);
 \filldraw[black] (1,1) circle (4pt);
 \filldraw[black] (0,2) circle (4pt);
 \draw (1,0-0.3) node[circle] {*};
 \draw[black] (1,3-0.3) node[circle] {\large{*}};
 \draw[rotate around={-45:(1.5,2.5)}] (1.5,2.5) ellipse (1.5 and 0.7);
\end{tikzpicture}
\rightarrow \!\!
\begin{tikzpicture}[scale=0.2, baseline={(0,0.15)}]
 \draw[black] (1,3) -- (1,2) -- (1,1) -- (1,0);
 \filldraw[black] (1,2) circle (4pt);
 \filldraw[black] (1,1) circle (4pt);
 \draw (1,0-0.3) node[circle] {*};
 \draw[black] (1,3-0.3) node[circle] {\large{*}};
 \draw (1,3) circle (0.6);
\end{tikzpicture}
$\\
\ \begin{tikzpicture}[scale=1]
 \draw[black] (0,0) -- (1,0) -- (1,1)--(3/2,1/5)--(1,0);
 \draw[black] (0,0) -- (1,1);
 \draw[black, densely dotted] (0,0) -- (1/2,1/5) -- (3/2,1/5);
 \draw[black, densely dotted] (1/2,1/5) -- (1,1);
 \filldraw[black] (3/2,1/5) circle (1.2pt);
\end{tikzpicture}\\
$
\begin{tikzpicture}[scale=0.2, baseline={(0,0.15)}]
 \draw[gray] (1,3) -- (0,2);
 \draw[black,thick] (0,2) -- (1,1) -- (1,0);
 \draw[gray] (1,3) -- (2,2) -- (1,1);
 \filldraw[gray] (2,2) circle (4pt);
 \filldraw[black] (1,1) circle (6pt);
 \filldraw[black] (0,2) circle (6pt);
 \draw[black] (1,0-0.3) node[circle] {\large{*}};
 \draw[gray] (1,3-0.3) node[circle] {*};
\end{tikzpicture}
\rightarrow \!\!
\begin{tikzpicture}[scale=0.2, baseline={(0,0.02)}]
 \draw[gray] (1,1) -- (1,0);
 \draw (1,0-0.3) node[circle] {\large{*}};
 \draw[gray] (1,1-0.3) node[circle] {*};
\end{tikzpicture}
$
\end{minipage}
\begin{minipage}{20mm}
\ \begin{tikzpicture}[scale=1]
 \fill[gray!40!white] (1,1) -- (3/2,1/5) -- (1/2,1/5);
 \draw[black] (0,0) -- (1,0) -- (1,1) --(3/2,1/5)--(1,0);
 \draw[black] (0,0) -- (1,1);
 \draw[black, densely dotted] (0,0) -- (1/2,1/5) -- (3/2,1/5);
 \draw[black, densely dotted] (1/2,1/5) -- (1,1);
\end{tikzpicture}\\
$
\begin{tikzpicture}[scale=0.2, baseline={(0,0.15)}]
 \draw[black] (1,3) -- (0,2);
 \draw[black, thick] (1,1) -- (0,2);
 \draw[black] (1,1) -- (1,0);
 \draw[black] (1,3) -- (2,2) -- (1,1);
 \filldraw[black] (0,2) circle (6pt);
 \filldraw[black] (1,1) circle (6pt);
 \filldraw[black] (2,2) circle (4pt);
 \draw (1,0-0.3) node[circle] {*};
 \draw (1,3-0.3) node[circle] {*};
 \draw[rotate around={-45:(0.5,1.5)}] (0.5,1.5) ellipse (1.5 and 0.7);
\end{tikzpicture}
\rightarrow \!\!
\begin{tikzpicture}[scale=0.2, baseline={(0,0.15)}]
 \draw[black] (1,3) -- (1,2) -- (1,1) -- (1,0);
 \filldraw[black] (1,2) circle (4pt);
 \filldraw[black] (1,1) circle (6pt);
 \draw (1,0-0.3) node[circle] {*};
 \draw (1,3-0.3) node[circle] {*};
 \draw (1,1) circle (0.6);
\end{tikzpicture}
$\\
\ \begin{tikzpicture}[scale=1]
 \draw[black] (0,0) -- (1,0) -- (1,1)--(3/2,1/5)--(1,0);
 \draw[black] (0,0) -- (1,1);
 \draw[black, densely dotted] (0,0) -- (1/2,1/5) -- (3/2,1/5);
 \draw[black, densely dotted] (1/2,1/5) -- (1,1);
 \filldraw[black] (0,0) circle (1.2pt);
\end{tikzpicture}\\
$
\begin{tikzpicture}[scale=0.2, baseline={(0,0.15)}]
 \draw[gray] (1,3) -- (0,2) -- (1,1);
 \draw[gray] (1,3) -- (2,2);
 \draw[black, thick] (2,2) -- (1,1) -- (1,0);
 \filldraw[black] (2,2) circle (6pt);
 \filldraw[black] (1,1) circle (6pt);
 \filldraw[gray] (0,2) circle (4pt);
 \draw[black] (1,0-0.3) node[circle] {\large{*}};
 \draw[gray] (1,3-0.3) node[circle] {*};
\end{tikzpicture}
\rightarrow \!\!
\begin{tikzpicture}[scale=0.2, baseline={(0,0.02)}]
 \draw[gray] (1,1) -- (1,0);
 \draw (1,0-0.3) node[circle] {\large{*}};
 \draw[gray] (1,1-0.3) node[circle] {*};
\end{tikzpicture}
$
\end{minipage}
\begin{minipage}{20mm}
\ \begin{tikzpicture}[scale=1]
 \fill[gray!40!white] (0,0) -- (1,1) -- (1/2,1/5);
 \draw[black] (0,0) -- (1,0) -- (1,1) --(3/2,1/5)--(1,0);
 \draw[black] (0,0) -- (1,1);
 \draw[black, densely dotted] (0,0) -- (1/2,1/5) -- (3/2,1/5);
 \draw[black, densely dotted] (1/2,1/5) -- (1,1);
\end{tikzpicture}\\
$
\begin{tikzpicture}[scale=0.2, baseline={(0,0.15)}]
 \draw[black] (1,3) -- (0,2) -- (1,1) -- (1,0);
 \draw[black] (1,3) -- (2,2);
 \draw[black, thick] (2,2) -- (1,1);
 \filldraw[black] (2,2) circle (6pt);
 \filldraw[black] (1,1) circle (6pt);
 \filldraw[black] (0,2) circle (4pt);
 \draw (1,0-0.3) node[circle] {*};
 \draw (1,3-0.3) node[circle] {*};
 \draw[rotate around={45:(1.5,1.5)}] (1.5,1.5) ellipse (1.5 and 0.7);
\end{tikzpicture}
\rightarrow \!\!
\begin{tikzpicture}[scale=0.2, baseline={(0,0.15)}]
 \draw[black] (1,3) -- (1,2) -- (1,1) -- (1,0);
 \filldraw[black] (1,2) circle (4pt);
 \filldraw[black] (1,1) circle (6pt);
 \draw (1,0-0.3) node[circle] {*};
 \draw (1,3-0.3) node[circle] {*};
 \draw (1,1) circle (0.6);
\end{tikzpicture}
$\\
\ \begin{tikzpicture}[scale=1]
 \draw[black] (0,0) -- (1,0) -- (1,1)--(3/2,1/5)--(1,0);
 \draw[black] (0,0) -- (1,1);
 \draw[black, densely dotted] (0,0) -- (1/2,1/5) -- (3/2,1/5);
 \draw[black, densely dotted] (1/2,1/5) -- (1,1);
 \filldraw[black] (1,0) circle (1.2pt);
\end{tikzpicture}\\
$
\begin{tikzpicture}[scale=0.2, baseline={(0,0.15)}]
 \draw[gray] (1,3) -- (0,2) -- (1,1);
 \draw[black,thick] (1,1)--(1,0);
 \draw[gray] (1,3) -- (2,2) -- (1,1);
 \filldraw[gray] (2,2) circle (4pt);
 \filldraw[black] (1,1) circle (6pt);
 \filldraw[gray] (0,2) circle (4pt);
 \draw[black] (1,0-0.3) node[circle] {\large{*}};
 \draw[gray] (1,3-0.3) node[circle] {*};
\end{tikzpicture}
\rightarrow \!\!
\begin{tikzpicture}[scale=0.2, baseline={(0,0.02)}]
 \draw[gray] (1,1) -- (1,0);
 \draw (1,0-0.3) node[circle] {\large{*}};
 \draw[gray] (1,1-0.3) node[circle] {*};
\end{tikzpicture}
$
\end{minipage}
\begin{minipage}{20mm}
\ \begin{tikzpicture}[scale=1]
 \fill[gray!40!white] (0,0) -- (1/2,1/5) -- (3/2,1/5) -- (1,0);
 \draw[black] (0,0) -- (1,0) -- (1,1) --(3/2,1/5)--(1,0);
 \draw[black] (0,0) -- (1,1);
 \draw[black, densely dotted] (0,0) -- (1/2,1/5) -- (3/2,1/5);
 \draw[black, densely dotted] (1/2,1/5) -- (1,1);
\end{tikzpicture}\\
$
\begin{tikzpicture}[scale=0.2, baseline={(0,0.15)}]
 \draw[black] (1,3) -- (0,2) -- (1,1);
 \draw[black, thick] (1,1) -- (1,0);
 \draw[black] (1,3) -- (2,2) -- (1,1);
 \filldraw[black] (2,2) circle (4pt);
 \filldraw[black] (1,1) circle (6pt);
 \filldraw[black] (0,2) circle (4pt);
 \draw[black] (1,0-0.3) node[circle] {\large{*}};
 \draw (1,3-0.3) node[circle] {*};
 \draw (1,0.5) ellipse (0.7 and 1.5);
\end{tikzpicture}
\rightarrow 
\begin{tikzpicture}[scale=0.2, baseline={(0,0.1)}]
 \draw[black] (1,2) -- (0,1) -- (1,0) -- (2,1) -- (1,2);
 \filldraw[black] (0,1) circle (4pt);
 \filldraw[black] (2,1) circle (4pt);
 \draw (1,2-0.3) node[circle] {*};
 \draw[black] (1,0-0.3) node[circle] {\large{*}};
 \draw (1,0) circle (0.6);
\end{tikzpicture}
$\\
\ \begin{tikzpicture}[scale=1]
 \draw[black] (0,0) -- (1,0) -- (1,1)--(3/2,1/5)--(1,0);
 \draw[black] (0,0) -- (1,1);
 \draw[black, densely dotted] (0,0) -- (1/2,1/5) -- (3/2,1/5);
 \draw[black, densely dotted] (1/2,1/5) -- (1,1);
 \filldraw[black] (1,1) circle (1.2pt);
 \draw (1,1) circle (4pt);
\end{tikzpicture}\\
$
\begin{tikzpicture}[scale=0.2, baseline={(0,0.15)}]
 \fill[gray!40!white] (1,3) -- (0,2) -- (1,1) -- (2,2) -- (1,3) -- cycle;
 \draw[gray] (1,3) -- (0,2) -- (1,1) -- (2,2) -- (1,3) -- cycle;
 \draw[black] (1,0) -- (1,1);
 \filldraw[gray] (2,2) circle (4pt);
 \filldraw[gray] (1,1) circle (4pt);
 \filldraw[gray] (0,2) circle (4pt);
 \draw[black] (1,0-0.3) node[circle] {\large{*}};
 \draw[gray] (1,3-0.3) node[circle] {*};
\end{tikzpicture}
\rightarrow \!\!
\begin{tikzpicture}[scale=0.2, baseline={(0,0.02)}]
 \draw[gray] (1,1) -- (1,0);
 \draw (1,0-0.3) node[circle] {\large{*}};
 \draw[gray] (1,1-0.3) node[circle] {*};
\end{tikzpicture}
$
\end{minipage}
\caption{One-to-one correspondence between faces and contractions}
\label{fig:faces}
\end{figure}
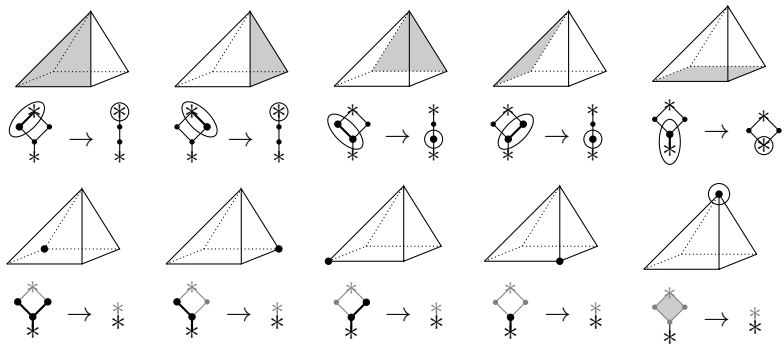

In Figure~\ref{fig:faces}, four facets of the order polytope $\Delta(P)$
are meeting at the same vertex circled.
Hence the corresponding point should be singular in $\bP_{\Delta(P)}$. 
In general, 
a singular locus comes from a contraction
replacing more inequalities to equalities than codimension.
A subposet $C\subset \hat P$ is said to be {\it convex} if it satisfies
\begin{align}
 u\in C, v \in C \text{ and } u\prec w \prec v  \text{ imply } w \in C.
\end{align}
Furthermore, let us call $C\subset \hat P$ a {\it minimal convex cycle}
if $C$ is (i) a full subposet not containing both $\hat 0$ and $\hat 1$, and 
(ii) a convex cycle such that all convex full subposets $C'\subset C$ are trees.
\begin{theorem}[{\cite[Corollary 2.4]{MR1382045}}]
 Let $P$ be a finite poset.
 An irreducible singular locus of $\bP_{\Delta(P)}$
 corresponds to a minimal convex cycle $C \subset \hat P$.
 For the corresponding contraction,
 all the fibers are singletons except one fiber $C$.
\end{theorem}

\begin{remark}
 \label{rem:hibi}
Now it is worth noting the homogeneous coordinate rings of Hibi toric varieties.
Since all the lattice points in $\Delta(P)$ are vertices, 
the homogeneous coordinate ring of $\bP_{\Delta(P)}$ is
a {\it Hibi algebra},
\begin{align}
 \label{eq:hcr}
 A_{J(P)} = \bC\left[ J(P)\right]/I_{J(P)},
\end{align}
where 
$\bC\left[J(P)\right]$ is the polynomial $\bC$-algebra in variables $p_\tau$ for $\tau\in J(P)$,
and $I_{J(P)}$ is the ideal coming from linear relations of vertices of $\Delta(P)$.
In fact, the ideal $I_{J(P)}$ has
quadratic generators,
\begin{align}
 p_\alpha p_\beta - 
 p_{\alpha\vee \beta} p_{\alpha \wedge \beta} \text{ for all } \alpha \not \sim \beta \in J(P).
\end{align}
Note here that $J(P)$ is a {\it lattice}, i.e., a poset 
with the least upper bound $\alpha \vee \beta$ 
and the greatest lower bound $\alpha \wedge \beta$ 
for each pair of elements.
In fact, $J(P)$ with the partial
order given by set inclusions is equipped with
$\alpha \vee \beta = \alpha \cup \beta$ 
and $\alpha \wedge \beta = \alpha \cap \beta$.
Furthermore, $J(P)$ becomes a {\it distributive lattice}, i.e.,
the lattice
with distributive laws,
\begin{align}
 \alpha \wedge (\beta \vee \gamma) = (\alpha \wedge \beta) \vee (\alpha \wedge \gamma)
 \quad \text{ for all } \alpha, \beta, \gamma \in J(P).
\end{align}
One may start from 
finite distributive lattices instead of finite posets,
which gives another description of Hibi toric varieties in literature.
\end{remark}
\begin{example}
 \label{ex:v2}
 Under the notation in Remark~\ref{rem:hibi},
 the Hibi toric variety $\bP_{\Delta(P)}$ in Example~\ref{ex:v} is 
 embedded as a quadric threefold in $\bP^4$ defined by
 \begin{align}
  p_{\begin{tikzpicture}[scale=0.1, baseline={(0,0.08)}]
     \draw[black] (0,1) -- (1,0);
     \filldraw[black] (0,1) circle (6pt);
     \filldraw[black] (1,0) circle (6pt);
   \end{tikzpicture}}\,
  p_{\begin{tikzpicture}[scale=0.1, baseline={(0,0.08)}]
     \draw[black] (0,0) -- (1,1);
     \filldraw[black] (0,0) circle (6pt);
     \filldraw[black] (1,1) circle (6pt);
    \end{tikzpicture}}
  -
  p_{\begin{tikzpicture}[scale=0.1, baseline={(0,0.08)}]
     \draw[black] (0,1) -- (1,0) -- (2,1);
     \filldraw[black] (0,1) circle (6pt);
     \filldraw[black] (1,0) circle (6pt);
     \filldraw[black] (2,1) circle (6pt);
   \end{tikzpicture}}\, 
  p_{\begin{tikzpicture}[scale=0.1, baseline={(0,0.08)}]
     \draw[black] (0,0);
     \filldraw[black] (0,0) circle (6pt);
    \end{tikzpicture}}
    =0.
 \end{align}
\end{example}

\subsection{Divisors}
Let $P$ be a finite poset.
For each edge $e\in E = \Edges (\hat P)$,
we have the corresponding invariant prime divisor,
denoted by $D_e$.
We write $D_{E'}=\sum_{e\in E'}D_e$ for each subset $E' \subset E$.
The linear equivalences in the divisor class group $\Cl(\bP_{\Delta(P)})$
are generated by the following relations:
\begin{align}
\label{eq:lin}
 \sum_{s(e)=u} D_e \simeq \sum_{t(e')=u} D_{e'} \quad \text{ for all } u \in P, 
\end{align}
where we write $t(e) \precdot s(e)$ for each $e = \left( s(e), t(e) \right) \in E \subset
\hat{P} \times \hat{P}$.

Let us describe the Picard group of $\bP_{\Delta(P)}$.
First, suppose $P$ is connected. 
For an order ideal $\tau \subset P$ and a subset $E' \subset E$,
a set of edges
\begin{align}
 \label{eq:etau}
 E'(\tau) = \left\{ e \in E' \mid t(e) \in \hat{0}\oplus \tau \text{ and } 
 s(e) \not\in \hat{0}\oplus\tau \right\}
\end{align}
defines a Weil divisor
$D_{E'(\tau)}$
on $\bP_{\Delta(P)}$.
We have $D_{E(\tau)} \simeq D_{E(P)}$ for all $\tau \in J(P)$.
The divisor $D_{E(P)}$ is, in fact, the Cartier divisor
corresponding to the lattice polytope $\Delta(P)$ itself.
One can show that the Picard group 
is isomorphic to $\bZ$ 
generated by the associated very ample invertible sheaf 
$
 \scO(1) = \scO(D_{E(P)}).
$

More generally, 
suppose $P = \sum_{j=1}^\rho P_j$ with $\rho$ connected components $P_1, \dots, P_\rho$.
Note that there is a natural decomposition as sets
\begin{align}
 E = \bigsqcup_{j=1}^\rho E_j,
\end{align}
where $E_j:= \Edges(\hat{P}_j) \subset E$ for each $j$.
The Picard group $\Pic(\bP_{\Delta(P)})$ becomes a free abelian group of rank $\rho$
generated by 
\begin{align}
\scL_j := \scO\left(D_{E_j(P_j)}\right)
\end{align}
for
each connected component $P_j \subset P$.
We have
\begin{align}
 \label{eq:O1}
 \scO(1) = \scO(D_{E(P)}) = \bigotimes_{j=1}^\rho \scL_j.
\end{align}

Next, we note a formula for the self-intersection number, 
\begin{align}
 \label{eq:degree}
 \left(D_{E(P)}^n\right) = n! \Vol \Delta(P)
 = c_{J(P)},
\end{align}
where $c_{J(P)}$ denotes the number of maximal chains on $J(P)$.
It follows from a formula for the Hilbert--Poincar\'e series of Hibi algebra $A_{J(P)}$
obtained by \cite[Corollary of Lemma 5]{MR790025}.

Lastly, let us suppose $P$ is {\it pure}.
Recall that a finite poset $P$ is called pure 
if every maximal chain on $P$ has the same length.
We define a height $h(u)$ of $u \in \hat P$ as the length of
the longest chain bounded above by $u$  in $\hat P$, and
write $h_P=h(\hat{1})$.
Thus an anti-canonical divisor $-K_{\bP_{\Delta(P)}} = D_E$ is written as
\begin{align}
 -K_{\bP_{\Delta(P)}} = \sum_{k=1}^{h_P} D_{E(\tau_k)} \simeq h_P D_{E(P)},
\end{align}
where 
$
 \tau_k := \left\{ u \in P \mid h(u) < k \right\} \in J(P)
$
for $k=1,\dots, h_P$.
Together with (\ref{eq:O1}), 
it turns out that
the Hibi toric variety $\bP_{\Delta(P)}$ 
for a pure poset $P$ is a Gorenstein Fano variety
with $\omega^\vee= \scO(-K_{\bP_{\Delta(P)}}) \simeq \scO(h_P)$.
Moreover, 
one can show that it has at worst terminal singularities,
by looking at the normal fan $\Sigma$ of $\Delta(P)$
(see \cite[Lemma 1.4]{MR2770546}).

\subsection{Small resolution}
Let $P$ be a finite pure poset.
The associated Hibi toric variety $\bP_{\Delta(P)}$
is a Gorenstein terminal Fano variety with $\omega^\vee \simeq \scO(h_P)$.
If $\bP_{\Delta(P)}$ is $\bQ$-factorial in addition,
it turns out to be non-singular, and even more, a product of projective spaces
by \cite[Corollary 2.4]{MR2770546}.
Even if it is not $\bQ$-factorial, 
we have the following property indicating the mildness of singularities of $\bP_{\Delta(P)}$.
\begin{proposition}
 \label{prop:small}
 For a finite pure poset $P$,
 any toric  crepant $\bQ$-factorialization
 of the Hibi toric variety $\bP_{\Delta(P)}$
 is a small resolution. 
\end{proposition}
\begin{proof}
 Let $P$ be a finite poset, $N = \bZ P$ and $M = \bZ^P$
 the free abelian groups of rank $n=\abs{P}$ dual to each other, 
 and $N_\bR =\bR P$ and $M_\bR = \bR^P$ the real scalar extensions, respectively.

 First, let us see a description of 
 the normal fan $\Sigma$ in $N_\bR$ for 
 the order polytope $\Delta(P) \subset M_\bR$.
 By definition, a one-dimensional cone in $\Sigma$ 
 is generated by the normal vector of a facet of $\Delta(P)$.
 Hence it corresponds to an edge of $\hat{P}$.
 Let $\delta(e) \in N$ denote such primitive vector 
 associated with $e \in E$,
 The map $\delta$ is extended
 to be the composite linear map
 $\delta = \mathrm{pr}\circ \partial \colon \bZ E \rightarrow N$ of
 \begin{align}
  \partial \colon \bZ E \rightarrow \bZ \hat{P}
  = N \oplus \bZ \hat 0 \oplus \bZ \hat 1, \quad  
  e \mapsto \partial(e):= t(e)-s(e)
 \end{align}
 and a projection $\mathrm{pr} \colon N \oplus \bZ \hat 0 \oplus \bZ \hat 1 \rightarrow N$.
 By using the same symbol $\delta$ as the real extension,
 each maximal dimensional cone in $\Sigma$ 
 associated with an order ideal $\tau\in J(P)$
 is written as
 \begin{align}
  \label{eq:maxcone}
  \sigma_\tau: = \Cone \delta\left(E \setminus E(\tau)\right) \subset N_\bR,
 \end{align}
 where $E(\tau)$ is defined by (\ref{eq:etau}).
 On the other hand, 
 $\Conv \delta(E)$ is a Gorenstein terminal Fano polytope by
 \cite[Lemma 1.3--1.5]{MR2770546}.
 Namely, for any $\tau \in J(P)$, all the primitive generators of the cone $\sigma_\tau$,
 i.e., the elements in $\delta(E\setminus E(\tau))$,
 lie on an affine hyperplane 
 with integral distance one from the origin, and it holds
 \begin{align}
  \label{eq:terminal}
  \left( \Conv \delta(E\setminus E(\tau) )\right) \cap N =\delta\left(E\setminus E(\tau)\right).
 \end{align}

Suppose $P$ is pure, and
let 
$X_{\widehat{\Sigma}} \rightarrow \mathbb P_{\Delta(P)}$
be a toric crepant $\bQ$-factorialization.
In other words, $\widehat{\Sigma}$ is a maximal simplicial refinement 
of the normal fan $\Sigma$ of $\Delta(P)$ such that $X_{\widehat{\Sigma}}$
denotes the corresponding $\bQ$-factorial toric variety.
Since $\bP_{\Delta(P)}$ has at worst terminal singularities,
the crepant birational morphism
$X_{\widehat{\Sigma}} \rightarrow \mathbb P_{\Delta(P)}$
is a small modification by definition.
Hence it is sufficient to show that $X_{\widehat{\Sigma}}$ is non-singular.

Fix a maximal dimensional cone $\sigma$ in $\widehat{\Sigma}$.
Since (\ref{eq:maxcone}) and (\ref{eq:terminal}),
there exist an order ideal $\tau \subset P$ and 
a subset 
$
 B\subset E\setminus  E(\tau)
$
consisting of $n+1$ elements
such that
$\sigma = \Cone \delta(B) \subset \sigma_\tau$.
As in the example shown in Figure~\ref{fig:small},
the subgraph $(\hat{P}, E\setminus E(\tau))$ of the Hasse diagram of $\hat{P}$
defining $\sigma_\tau$ consists of two connected graphs, and
the subgraph $(\hat{P}, B)$
defining $\sigma$ consists of two connected tree graphs.
In fact, if $(\hat{P},B)$ contains a cycle,
$\sigma$ cannot have maximal dimension.
Therefore, we have a unique unoriented path in $(\hat{P},B)$
from any $u\in P$ to $\hat 0$ or $\hat 1$, 
which attains a value $\pm u \in \bZ \delta(B)$ by summing up and mapping by $\delta$.
Hence $\delta(B)$ forms a $\bZ$-basis of $N = \bZ P$.
Since $\sigma$ is arbitrary, it follows that
$X_{\widehat{\Sigma}}$ is non-singular.
\begin{figure}[H]
 \centering
\begin{minipage}{20mm}
 \ \! $E$\\
 \begin{tikzpicture}[scale=0.32, baseline={(0,0.4)}]
  \draw[black] (0,1) -- (0,2) -- (1,1) -- (2,2) -- (2,1) -- (1,2) -- (0,1) -- cycle;
  \draw[black] (1,0) -- (0,1);
  \draw[black] (1,0) -- (1,1);
  \draw[black] (1,0) -- (2,1);
  \draw[black] (1,3) -- (0,2);
  \draw[black] (1,3) -- (1,2);
  \draw[black] (1,3) -- (2,2);
  \filldraw[black] (0,1) circle (2.5pt);
  \filldraw[black] (0,2) circle (2.5pt);
  \filldraw[black] (1,1) circle (2.5pt);
  \filldraw[black] (1,2) circle (2.5pt);
  \filldraw[black] (2,1) circle (2.5pt);
  \filldraw[black] (2,2) circle (2.5pt);
  \draw (1,0-0.1875) node[circle] {*};
  \draw (1,3-0.1875) node[circle] {*};
 \end{tikzpicture}
\end{minipage}
\begin{minipage}{20mm}
 \hspace{-6pt}$E\setminus E(\tau)$\\
 \begin{tikzpicture}[scale=0.32, baseline={(0,0.4)}]
  \draw[black, thick] (0,1) -- (0,2) -- (1,3) -- (1,2) -- (0,1) -- cycle;
  \draw[black, thick] (2,1) -- (2,2) -- (1,1) -- (1,0) -- (2,1) -- cycle;
  \draw[black, densely dotted] (0,2) -- (1,1);
  \draw[black, densely dotted] (0,1) -- (1,0);
  \draw[black, densely dotted] (1,2) -- (2,1);
  \draw[black, densely dotted] (1,3) -- (2,2);
  \filldraw[black] (0,1) circle (2.5pt);
  \filldraw[black] (0,2) circle (2.5pt);
  \filldraw[black] (1,1) circle (4pt);
  \filldraw[black] (1,2) circle (2.5pt);
  \filldraw[black] (2,1) circle (4pt);
  \filldraw[black] (2,2) circle (4pt);
  \draw[black] (1,0-0.1875) node[circle] {\large{*}};
  \draw (1,3-0.1875) node[circle] {*};
 \end{tikzpicture}
\end{minipage}
\begin{minipage}{10mm}
 \ \! $B$\\
 \begin{tikzpicture}[scale=0.32, baseline={(0,0.4)}]
  \draw[black, thick] (0,1) -- (0,2) -- (1,3) -- (1,2);
  \draw[black, thick] (2,2) -- (1,1) -- (1,0) -- (2,1);
  \draw[black, densely dotted] (0,1) -- (1,2);
  \draw[black, densely dotted] (2,2) -- (2,1);
  \draw[black, densely dotted] (0,2) -- (1,1);
  \draw[black, densely dotted] (0,1) -- (1,0);
  \draw[black, densely dotted] (1,2) -- (2,1);
  \draw[black, densely dotted] (1,3) -- (2,2);
  \filldraw[black] (0,1) circle (2.5pt);
  \filldraw[black] (0,2) circle (2.5pt);
  \filldraw[black] (1,1) circle (2.5pt);
  \filldraw[black] (1,2) circle (2.5pt);
  \filldraw[black] (2,1) circle (2.5pt);
  \filldraw[black] (2,2) circle (2.5pt);
  \draw[black] (1,0-0.1875) node[circle] {*};
  \draw (1,3-0.1875) node[circle] {*};
 \end{tikzpicture} 
\end{minipage}
 \caption{An example of subgraphs corresponding to $\sigma_\tau$ and $\sigma$}
 \label{fig:small}
\end{figure}
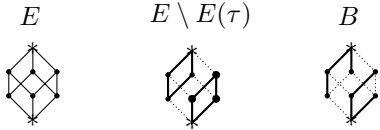
\end{proof}


\section{CICY threefolds in Hibi toric varieties}
\subsection{Examples}
We describe Calabi--Yau threefolds 
obtained as a complete intersection
of general sections of invertible sheaves
in Hibi toric varieties.
We call such Calabi--Yau threefolds 
{\it complete intersection Calabi--Yau (CICY)}
threefolds in Hibi toric varieties.

Let $P$ be a finite poset, 
and $X_0$ a CICY threefold in $\bP_{\Delta(P)}$.
From the adjunction formula,
$\bP_{\Delta(P)}$
has at worst Gorenstein singularities.
In other words, all connected components of $P$ are pure.
If $P$ is a disjoint union of several pure connected posets,
we have a number of Calabi--Yau threefolds
as complete intersections of nef divisors in $\bP_{\Delta(P)}$,
e.g.,
complete intersection Calabi--Yau threefolds in
products of projective spaces.
However, we assume in the sequel 
that $P$ is pure connected for simplicity.
Under this assumption,
$X_0$ is merely a complete intersection of very ample divisors in $\bP_{\Delta(P)}$.
Let $(d_1, \dots, d_r) \subset \bP_{\Delta(P)}$ denote
a complete intersection variety of 
very ample divisors defined by 
general sections
of $\scO(d_1) , \dots, \scO(d_r)$, respectively.
Then $X_0 = (d_1, \dots, d_r) \subset \bP_{\Delta(P)}$ is 
a CICY threefold if and only if
\begin{align}
 \label{eq:cycond}
 \sum_{j=1}^r d_j = h_P \ \text{ and }\  \abs{P}-r=3.
\end{align}
\begin{example}
 The poset in (\ref{eq:N}) gives an example of 
 hypersurface Calabi--Yau threefolds in Hibi toric varieties.
 Thus $X_0 = (d_1=3) \subset \bP_{\Delta(P)}$ in this case.
\end{example}

\begin{example}
 \label{ex:cube}
 As an example to illustrate calculations,
 we introduce a finite pure connected poset $P=\scP_1$:
 \begin{align}
  \begin{tikzpicture}[scale=0.4, baseline={(0,0.1)}]
  \draw[black] (0,0) -- (0,1) -- (1,0) -- (2,1) -- (2,0) -- (1,1) -- (0,0) -- cycle;
  \filldraw[black] (0,0) circle (2pt);
  \filldraw[black] (0,1) circle (2pt);
  \filldraw[black] (1,0) circle (2pt);
  \filldraw[black] (1,1) circle (2pt);
  \filldraw[black] (2,0) circle (2pt);
  \filldraw[black] (2,1) circle (2pt);
 \end{tikzpicture}
 \end{align}
 We have $\abs{\mathcal{P}_1}=6$ and $h_{\mathcal{P}_1}=3$, and hence,
 the associated Hibi toric variety $\bP_{\Delta(\mathcal{P}_1)}$
 is a six-dimensional Gorenstein terminal Fano variety with 
 $\omega^\vee \simeq \scO(3)$.
 We have a linear section Calabi--Yau threefold 
 $X_0 = (1^3) \subset \bP_{\Delta(\mathcal{P}_1)}$.

 The first part of $J(\scP_1)$ corresponds to order ideals in the
 left of Figure~\ref{fig:jp1}.
 By continuing while focusing on set inclusions,
 we obtain the lattice $J(\scP_1)$ as in the middle of Figure~\ref{fig:jp1},
 consisting of $\abs{J(\scP_1)}=18$ elements.
 Moreover, we have $c_{J(\scP_1)} = 48$, the number of maximal chains on $J(\scP_1)$,
 by counting as in the
 right in Figure~\ref{fig:jp1}.
 \begin{figure}[H]
 \label{fig:p1}
 \centering
 \begin{minipage}{3cm}
 \centering
 \mbox{
 \begin{tikzpicture}[scale=0.2]
  \draw[black] (0,0) -- (0,1) -- (1,0) -- (2,1) -- (2,0) -- (1,1) -- (0,0) -- cycle;
  \filldraw[black] (0,0) circle (4pt);
  \filldraw[black] (0,1) circle (4pt);
  \filldraw[black] (1,0) circle (4pt);
  \filldraw[black] (1,1) circle (4pt);
  \filldraw[black] (2,0) circle (4pt);
  \filldraw[black] (2,1) circle (4pt);
 \end{tikzpicture}
 }\\
 \mbox{
  \begin{tikzpicture}[scale=0.2]
  \draw[black] (1,0) -- (2,1) -- (2,0) -- (1,1) -- (0,0);
  \filldraw[black] (0,0) circle (4pt);
  \filldraw[black] (1,0) circle (4pt);
  \filldraw[black] (1,1) circle (4pt);
  \filldraw[black] (2,0) circle (4pt);
  \filldraw[black] (2,1) circle (4pt);
 \end{tikzpicture}
 \begin{tikzpicture}[scale=0.2]
  \draw[black] (0,0) -- (0,1) -- (1,0) -- (2,1) -- (2,0);
  \filldraw[black] (0,0) circle (4pt);
  \filldraw[black] (0,1) circle (4pt);
  \filldraw[black] (1,0) circle (4pt);
  \filldraw[black] (2,0) circle (4pt);
  \filldraw[black] (2,1) circle (4pt);
 \end{tikzpicture}
 \begin{tikzpicture}[scale=0.2]
  \draw[black] (2,0) -- (1,1) -- (0,0) -- (0,1) -- (1,0);
  \filldraw[black] (0,0) circle (4pt);
  \filldraw[black] (0,1) circle (4pt);
  \filldraw[black] (1,0) circle (4pt);
  \filldraw[black] (1,1) circle (4pt);
  \filldraw[black] (2,0) circle (4pt);
 \end{tikzpicture}
 }\\
 \mbox{
 \begin{tikzpicture}[scale=0.2]
  \draw[black] (1,0) -- (2,1) -- (2,0);
  \filldraw[black] (0,0) circle (4pt);
  \filldraw[black] (1,0) circle (4pt);
  \filldraw[black] (2,0) circle (4pt);
  \filldraw[black] (2,1) circle (4pt);
 \end{tikzpicture}
 \begin{tikzpicture}[scale=0.2]
  \draw[black] (2,0) -- (1,1) -- (0,0);
  \filldraw[black] (0,0) circle (4pt);
  \filldraw[black] (1,0) circle (4pt);
  \filldraw[black] (1,1) circle (4pt);
  \filldraw[black] (2,0) circle (4pt);
 \end{tikzpicture}
 \begin{tikzpicture}[scale=0.2]
  \draw[black] (0,0) -- (0,1) -- (1,0);
  \filldraw[black] (0,0) circle (4pt);
  \filldraw[black] (0,1) circle (4pt);
  \filldraw[black] (1,0) circle (4pt);
  \filldraw[black] (2,0) circle (4pt);
 \end{tikzpicture}
 }\\
 $\vdots$
\end{minipage}
\quad
\begin{tikzpicture}[scale=0.4, baseline = {(0.4,1)}]
  \draw[black] (0,1) -- (0,2) -- (1,1) -- (2,2) -- (2,1) -- (1,2) -- (0,1) -- cycle;
  \draw[black] (0,4) -- (0,5) -- (1,4) -- (2,5) -- (2,4) -- (1,5) -- (0,4) -- cycle;
  \draw[black] (1,0) -- (0,1);
  \draw[black] (1,0) -- (1,1);
  \draw[black] (1,0) -- (2,1);
  \draw[black] (0,4) -- (1,3) -- (0,2);
  \draw[black] (1,4) -- (1,3) -- (1,2);
  \draw[black] (2,4) -- (1,3) -- (2,2);
  \draw[black] (1,6) -- (0,5);
  \draw[black] (1,6) -- (1,5);
  \draw[black] (1,6) -- (2,5);
  \draw[black] (0,4) to [out=200,in=90] (-0.5,3) to [out=-90,in=160] (0,2);
  \draw[black] (1,4) to [out=200,in=90] (0.5,3) to [out=-90,in=160] (1,2);
  \draw[black] (2,4) to [out=-20,in=90] (2.5,3) to [out=-90,in=20] (2,2);
  \filldraw[black] (-0.5,3) circle (2pt);
  \filldraw[black] (0,1) circle (2pt);
  \filldraw[black] (0,2) circle (2pt);
  \filldraw[black] (0,4) circle (2pt);
  \filldraw[black] (0,5) circle (2pt);
  \filldraw[black] (0.5,3) circle (2pt);
  \filldraw[black] (1,0) circle (2pt);
  \filldraw[black] (1,1) circle (2pt);
  \filldraw[black] (1,2) circle (2pt);
  \filldraw[black] (1,3) circle (2pt);
  \filldraw[black] (1,4) circle (2pt);
  \filldraw[black] (1,5) circle (2pt);
  \filldraw[black] (1,6) circle (2pt);
  \filldraw[black] (2,1) circle (2pt);
  \filldraw[black] (2,2) circle (2pt);
  \filldraw[black] (2,4) circle (2pt);
  \filldraw[black] (2,5) circle (2pt);
  \filldraw[black] (2.5,3) circle (2pt);
 \end{tikzpicture} \quad \quad
\begin{tikzpicture}[scale=0.4, baseline = {(0.4,1)}]
  \draw[black] (0,1) -- (0,2) -- (1,1) -- (2,2) -- (2,1) -- (1,2) -- (0,1) -- cycle;
  \draw[black] (0,4) -- (0,5) -- (1,4) -- (2,5) -- (2,4) -- (1,5) -- (0,4) -- cycle;
  \draw[black] (1,0) -- (0,1);
  \draw[black] (1,0) -- (1,1);
  \draw[black] (1,0) -- (2,1);
  \draw[black] (0,4) -- (1,3) -- (0,2);
  \draw[black] (1,4) -- (1,3) -- (1,2);
  \draw[black] (2,4) -- (1,3) -- (2,2);
  \draw[black] (1,6) -- (0,5);
  \draw[black] (1,6) -- (1,5);
  \draw[black] (1,6) -- (2,5);
  \draw[black] (0,4) to [out=200,in=90] (-0.5,3) to [out=-90,in=160] (0,2);
  \draw[black] (1,4) to [out=200,in=90] (0.5,3) to [out=-90,in=160] (1,2);
  \draw[black] (2,4) to [out=-20,in=90] (2.5,3) to [out=-90,in=20] (2,2);
  \filldraw[black] (-0.5,3) circle (2pt);
  \filldraw[black] (0,1) circle (2pt);
  \filldraw[black] (0,2) circle (2pt);
  \filldraw[black] (0,4) circle (2pt);
  \filldraw[black] (0,5) circle (2pt);
  \filldraw[black] (0.5,3) circle (2pt);
  \filldraw[black] (1,0) circle (2pt);
  \filldraw[black] (1,1) circle (2pt);
  \filldraw[black] (1,2) circle (2pt);
  \filldraw[black] (1,3) circle (2pt);
  \filldraw[black] (1,4) circle (2pt);
  \filldraw[black] (1,5) circle (2pt);
  \filldraw[black] (1,6) circle (2pt);
  \filldraw[black] (2,1) circle (2pt);
  \filldraw[black] (2,2) circle (2pt);
  \filldraw[black] (2,4) circle (2pt);
  \filldraw[black] (2,5) circle (2pt);
  \filldraw[black] (2.5,3) circle (2pt);
  \draw (1+0.3,6+0.1) node[circle] {\tiny{$1$}};
  \draw (-0.3,5+0.1) node[circle] {\tiny{$1$}};
  \draw (1+ 0.3,5+0.1) node[circle] {\tiny{$1$}};
  \draw (2+ 0.3,5+0.1) node[circle] {\tiny{$1$}};
  \draw (-0.3,4+0.2) node[circle] {\tiny{$2$}};
  \draw (1+ 0.3,4-0.1) node[circle] {\tiny{$2$}};
  \draw (2+ 0.3,4+0.2) node[circle] {\tiny{$2$}};
  \draw (-0.5-0.3,3) node[circle] {\tiny{$2$}};
  \draw (0.5-0.3,3) node[circle] {\tiny{$2$}};
  \draw (2.5+0.3,3) node[circle] {\tiny{$2$}};
  \draw (1+ 0.5,3) node[circle] {\tiny{$6$}};
  \draw (-0.3,2-0.2) node[circle] {\tiny{$8$}};
  \draw (1+ 0.3,2+0.1) node[circle] {\tiny{$8$}};
  \draw (2+ 0.3,2-0.2) node[circle] {\tiny{$8$}};
  \draw (-0.45,1-0.1) node[circle] {\tiny{$16$}};
  \draw (1+ 0.4,1-0.1) node[circle] {\tiny{$16$}};
  \draw (2+ 0.45,1-0.1) node[circle] {\tiny{$16$}};
  \draw (1+ 0.4,-0.1) node[circle] {\tiny{$48$}};
 \end{tikzpicture}
 \caption{On the Hasse diagram of $J(\scP_1)$}
 \label{fig:jp1}
 \end{figure}
\end{example}

\subsection{Stringy Hodge numbers}
Let $P$ be a pure connected poset and $X_0=(d_1,\dots,d_r) \subset \bP_{\Delta(P)}$.
We have a small resolution $Y\rightarrow X_0$, by taking the strict transform of $X_0$
for a small toric resolution $X_{\widehat{\Sigma}}\rightarrow \bP_{\Delta(P)}$ for example.
In this case,
the stringy Hodge numbers of $X_0$ are nothing but
usual Hodge numbers of $Y$.
From \cite[Proposition 8.6]{MR1463173},
the following combinatorial formulas hold.
\begin{proposition}
\begin{align}
 \label{eq:h11}
 h_{\mathrm{st}}^{1,1}(X_0)=h^{1,1}(Y)=&\abs E- \abs P,\\
 \begin{split}
  \label{eq:h21}
 h_{\mathrm{st}}^{1,2}(X_0)=h^{1,2}(Y)=&
 \sum_{i\in [r]}\left[ \sum_{J \subset [r]}(-1)^{\abs J} 
	l\left((d_i-d_J)\Delta(P) \right)\right] \\ 
        &- \sum_{J \subset [r]}(-1)^{r-\abs J}\left[ \sum_{e\in E}
        l^*(d_J\theta_e) \right]-\abs P,
 \end{split}
\end{align}
where $l(\theta)$ and $l^*(\theta)$
denote the number of lattice points in a face $\theta \subset M_\bR$
and in the interior of $\theta$, respectively;
$[r]=\left\{1, \dots, r\right\}$, 
$d_J=\sum_{j\in J}d_j$ and $\theta_e$ 
is the facet of $\Delta(P)$ corresponding to an edge $e\in E$.
\end{proposition}
Note that a nonzero contribution in the first term of 
(\ref{eq:h21})
comes only from the range of $d_i-d_J\ge 0$,
and in the second term it comes only from the range $d_J = h_P-1$ or $h_P$.
In particular, if (i) $d_j=1$ for all $j$, and (ii)
$P$ has no ordinal summand of singleton, 
i.e., $P\ne P_1 \oplus u \oplus P_2$ for any $P_1$ and $P_2$, we have
\begin{align}
 \label{eq:h21simplify}
 h_{\mathrm{st}}^{1,2}(X_0) = h^{1,2}(Y)=
 h_P \left( \abs{J(P)} - h_P  \right) -\sum_{e \in E} l^*\left( h_P \theta_e \right)
 -\abs{P}.
\end{align}

\begin{example}
 For the example $P= \scP_1$
 and a complete intersection $X_0=(1^3)$,
 we obtain $h_{\mathrm{st}}^{1,1}(X_0) = 12-6 = 6$ from (\ref{eq:h11}).
 Since the both conditions (i) and (ii) are satisfied, 
 the simplified formula (\ref{eq:h21simplify}) holds in this case.
 By the symmetry $\frak{S}_3\times\frak{S}_3$ of $\hat{\scP}_1$ as a poset
 and the order duality, 
 it suffices to see
 the following two types of facets:
\begin{align}
 \begin{tikzpicture}[scale=0.32, baseline={(0,0.4)}]
  \draw[black] (0,1) -- (0,2) -- (1,1) -- (2,2) -- (2,1) -- (1,2) -- (0,1) -- cycle;
  \draw[black] (1,0) -- (0,1);
  \draw[black] (1,0) -- (1,1);
  \draw[black] (1,0) -- (2,1);
  \draw[black, thick] (1,3) -- (0,2);
  \draw[black] (1,3) -- (1,2);
  \draw[black] (1,3) -- (2,2);
  \filldraw[black] (0,1) circle (2.5pt);
  \filldraw[black] (0,2) circle (4pt);
  \filldraw[black] (1,1) circle (2.5pt);
  \filldraw[black] (1,2) circle (2.5pt);
  \filldraw[black] (2,1) circle (2.5pt);
  \filldraw[black] (2,2) circle (2.5pt);
  \draw (1,0-0.1875) node[circle] {*};
  \draw[black] (1,3-0.1875) node[circle] {\large{*}};
  \draw[rotate around={45:(0.5,2.5)}] (0.5,2.5) ellipse (1.2 and 0.45);
 \end{tikzpicture} \rightarrow
 \begin{tikzpicture}[scale=0.2,baseline={(0,0.2)}]
  \draw[black] (0,1) -- (1,2) -- (2,3) -- (3,2) -- (4,1) -- (2,0) -- (0,1) -- cycle;
  \draw[black] (1,2) -- (2,1) -- (2,0);
  \draw[black] (3,2) -- (2,1) -- (2,0);
  \filldraw[black] (0,1) circle (4pt);
  \filldraw[black] (1,2) circle (4pt);
  \filldraw[black] (2,1) circle (4pt);
  \filldraw[black] (3,2) circle (4pt);
  \filldraw[black] (4,1) circle (4pt);
  \draw (2,0-0.3) node[circle] {*};
  \draw[black] (2,3-0.3) node[circle] {\large{*}};
  \draw[black] (2,3) circle (0.6);
 \end{tikzpicture} \quad  \text{ and } \quad  \ 
 \begin{tikzpicture}[scale=0.32, baseline={(0,0.4)}]
  \draw[black, thick] (0,1) -- (0,2);
  \draw[black] (0,2) -- (1,1) -- (2,2) -- (2,1) -- (1,2) -- (0,1);
  \draw[black] (1,0) -- (0,1);
  \draw[black] (1,0) -- (1,1);
  \draw[black] (1,0) -- (2,1);
  \draw[black] (1,3) -- (0,2);
  \draw[black] (1,3) -- (1,2);
  \draw[black] (1,3) -- (2,2);
  \filldraw[black] (0,1) circle (4pt);
  \filldraw[black] (0,2) circle (4pt);
  \filldraw[black] (1,1) circle (2.5pt);
  \filldraw[black] (1,2) circle (2.5pt);
  \filldraw[black] (2,1) circle (2.5pt);
  \filldraw[black] (2,2) circle (2.5pt);
  \draw (1,0-0.1875) node[circle] {*};
  \draw (1,3-0.1875) node[circle] {*};
  \draw (0,1.5) ellipse (0.4 and 1.1);
 \end{tikzpicture} \rightarrow
 \begin{tikzpicture}[scale=0.2, baseline={(0,0.3)}]
  \draw[black] (2,1) -- (1,0) -- (0,1) -- (2,3) -- (2,1) -- (0,3) -- (1,4) -- (2,3);
  \draw[black] (0,3) to [out=200,in=90] (-0.5,2) to [out=-90,in=160] (0,1);
  \filldraw[black] (-0.5,2) circle (6pt);
  \draw[black] (-0.5, 2) circle (0.6);
  \filldraw[black] (0,1) circle (4pt);
  \filldraw[black] (0,3) circle (4pt);
  \filldraw[black] (2,1) circle (4pt);
  \filldraw[black] (2,3) circle (4pt);
  \draw (1,0-0.3) node[circle] {*};
  \draw (1,4-0.3) node[circle] {*};
\end{tikzpicture}.
\end{align}
 There are six facets for each type, and clearly
 $l^*(3 \theta) = 1$ (resp.\ $0$) 
 for the former (resp.\ the latter) type.
 Therefore $h_{\mathrm{st}}^{1,2}(X_0) = 3(18-3) - 6 - 6 = 33$.
\end{example}

\subsection{Numbers of nodes}
Recall that three-dimensional Gorenstein terminal toric singularities 
are at worst nodes  (i.e, ordinary double points), 
since they are presented by three-dimensional cones over a unit triangle or a unit square.
Together with the Bertini-type theorem for toroidal singularities,
the singularities of $X_0$ are also at worst nodes.
We count the number of nodes $\mathrm{dp}(X_0)$ on $X_0$ in the following.

Each node on $X_0$
lies on one singular locus of codimension three
of $\bP_{\Delta(P)}$,
corresponding to
a minimal convex cycle
$C \subset \hat{P}$
with four elements.
There are four types of such minimal convex cycles:
\begin{align}
 \begin{tikzpicture}[scale=0.2, baseline={(0.2,0.1)}]
  \fill[gray!40!white] (0,1) -- (1,2) -- (2,1) -- (1,0) -- (0,1) -- cycle;
  \draw[black] (0,1) -- (1,2) -- (2,1) -- (1,0) -- (0,1) -- cycle;
  \filldraw[black] (0,1) circle (4pt);
  \filldraw[black] (1,2) circle (4pt);
  \filldraw[black] (2,1) circle (4pt);
  \filldraw[black] (1,0) circle (4pt);
 \end{tikzpicture}, \ 
 \begin{tikzpicture}[scale=0.2, baseline={(0.2,0.1)}]
  \fill[gray!40!white] (0,1) -- (1,2) -- (2,1) -- (1,0) -- (0,1) -- cycle;
  \draw[black] (0,1) -- (1,2) -- (2,1) -- (1,0) -- (0,1) -- cycle;
  \filldraw[black] (0,1) circle (4pt);
  \filldraw[black] (2,1) circle (4pt);
  \filldraw[black] (1,0) circle (4pt);
  \draw (1,2-0.3) node[circle] {*};
 \end{tikzpicture}, \ 
 \begin{tikzpicture}[scale=0.2, baseline={(0.2,0.1)}]
  \fill[gray!40!white] (0,1) -- (1,2) -- (2,1) -- (1,0) -- (0,1) -- cycle;
  \draw[black] (0,1) -- (1,2) -- (2,1) -- (1,0) -- (0,1) -- cycle;
  \filldraw[black] (0,1) circle (4pt);
  \filldraw[black] (1,2) circle (4pt);
  \filldraw[black] (2,1) circle (4pt);
  \draw (1,0-0.3) node[circle] {*};
 \end{tikzpicture} \  \text{ or } \ \ 
 \begin{tikzpicture}[scale=0.2, baseline={(0.2,0.1)}]
  \fill[gray!40!white] (0,0) -- (0,2) -- (1,1) -- (0,0) -- cycle;
  \fill[gray!80!white] (2,2) -- (1,1) -- (2,0) -- (2,2) -- cycle;
  \draw[black] (0,0) -- (0,2) -- (2,0) -- (2,2) -- (0,0) -- cycle;
  \filldraw[black] (0,0) circle (4pt);
  \filldraw[black] (0,2) circle (4pt);
  \filldraw[black] (2,0) circle (4pt);
  \filldraw[black] (2,2) circle (4pt);
 \end{tikzpicture}\, .
\end{align}
Let $\Lambda_4(\hat P)$ denote the set of 
such minimal convex cycles consisting of four elements on $\hat P$.
For each $C\in \Lambda_4(\hat P)$,
there is the contraction $\hat P\rightarrow \hat P_C$ such that
all the fibers are singleton except one fiber $C$.
Of course it holds $\abs{P_C} = \abs{P} -3$ for all $C\in \Lambda_4(\hat P)$.
Hence 
$C$ defines a singular locus of codimension three and of
degree $\deg \Delta(P_C) = c_{J(P_C)}$ from (\ref{eq:degree}).
Therefore, the number of nodes $\mathrm{dp}(X_0)$
is calculated by a formula
\begin{align}
 \label{eq:nodes}
 \mathrm{dp}(X_0) = \prod_{j=1}^r d_j \sum_{C\in \Lambda_4(\hat{P})}c_{J(P_C)}.
\end{align}

\begin{example}
 There are six minimal convex cycles on $\hat{\scP_1}$, 
 each of which consists of four elements.
 By symmetry, they are all equivalent to
\begin{align}
 \label{eq:cycle}
 \begin{tikzpicture}[scale=0.32, baseline={(0,0.4)}]
  \fill[gray!40!white!] (1,0) -- (0,1) -- (0,2) -- (1,1) -- (1,0) -- cycle;
  \draw[black] (0,1) -- (0,2) -- (1,1) -- (2,2) -- (2,1) -- (1,2) -- (0,1) -- cycle;
  \draw[black] (1,0) -- (0,1);
  \draw[black] (1,0) -- (1,1);
  \draw[black] (1,0) -- (2,1);
  \draw[black] (1,3) -- (0,2);
  \draw[black] (1,3) -- (1,2);
  \draw[black] (1,3) -- (2,2);
  \filldraw[black] (0,1) circle (2.5pt);
  \filldraw[black] (0,2) circle (2.5pt);
  \filldraw[black] (1,1) circle (2.5pt);
  \filldraw[black] (1,2) circle (2.5pt);
  \filldraw[black] (2,1) circle (2.5pt);
  \filldraw[black] (2,2) circle (2.5pt);
  \draw (1,0-0.15) node[circle] {*};
  \draw (1,3-0.15) node[circle] {*};
 \end{tikzpicture}
 \rightarrow
 \begin{tikzpicture}[scale=0.2, baseline={(0,0.2)}]
  \draw[black] (1,3) -- (0,2) -- (1,1) -- (2,2) -- (1,3) -- cycle;
  \draw[black] (1,0) -- (1,1);
  \filldraw[black] (0,2) circle (4pt);
  \filldraw[black] (1,1) circle (4pt);
  \filldraw[black] (2,2) circle (4pt);
  \draw (1,3-0.3) node[circle] {*};
  \draw (1,0-0.3) node[circle] {*};
\end{tikzpicture}.
\end{align}
 Since the locus is a quadric threefold in Example~\ref{ex:v2},
 we obtain $\mathrm{dp}(X_0)=6\cdot 2 = 12$. 
\end{example}

\subsection{Smoothability}
For smoothability, we follow the argument in the case of hypersurfaces in
toric varieties by \cite{MR2801412}.
Let 
$\left\{ p_1, \dots, p_\mathrm{dp} \right\}$ be the set of nodes on $X_0$,
where $\mathrm{dp}=\mathrm{dp}(X_0)$, and
$f: Y \rightarrow X_0$ be a small resolution.
The exceptional lines $L_i := f^{-1}(p_i) \simeq \bP^1$ for $i=1,\dots, \mathrm{dp}$
form a linear subspace of $H_2(Y,\bC)$.
By \cite[Theorem 2.5]{MR1923221}, the Calabi--Yau threefold $X_0$ is smoothable
by a flat deformation
if and only if
the homology classes $[L_i]\in H_2(Y,\bC)$ satisfy a relation,
\begin{align}
 \sum_{i=1}^{\mathrm{dp}} \alpha_i [L_i] = 0, 
\end{align}
where $\alpha_i \ne 0$ for all $i = 1,\dots, \mathrm{dp}$.
 
Note that one can identify
\begin{align}
 \label{eq:id}
 H_2(Y,\bQ)\simeq \left\{ (\lambda_e)_{e\in E} \,
  \middle| \, 
\sum_{e\in E} \lambda_e \delta(e) = 0
\right\} \subset \bQ^E.
\end{align}
Under this identification,
the homology class $[L_i]$ coincides with
a relation,
\begin{align}
 \rho_C \ \colon \  \delta(e_p) + \delta(e_q) - \delta(e_r) - \delta(e_s) = 0
\end{align}
up to signs,
where the corresponding node $p_i$ lies on a singular locus 
associated with a minimal convex cycle $C\in\Lambda_4(\hat{P})$,
and the cycle $C$ with an orientation
passes through the four edges;
$e_p, e_q$ in the forward direction and  $e_r, e_s$ in the opposite direction.

\begin{proposition}
 \label{prop:smoothing}
 Let $P$ be a pure connected poset,
 $X_0$ a CICY threefold
 of degree $(d_1,\dots, d_r)$
 in the Hibi toric variety $\bP_{\Delta(P)}$.
 If $\prod_{j=1}^r d_j > 1$, $X_0$ is smoothable.
 If $\prod_{j=1}^r d_j = 1$,
 $X_0$ is smoothable if and only if
 for any $C\in \Lambda_4(\hat{P})$ such that $P_C$ is a chain
 the element $\rho_C$ is a linear combination of the
 remaining elements $\rho_{C'}$ 
 with $C' \in \Lambda_4(\hat{P})$, $C\ne C'$.
\end{proposition}

\begin{example}
 For the example $P=\scP_1$,
 $X_0=(1^3)$ is smoothable since $P_C$ is not a chain for all $C\in \Lambda_4(\hat{\scP}_1)$
 as we see in (\ref{eq:cycle}),
 although $\prod_{j=1}^3 d_j = 1$.
\end{example}
\begin{example}
 There are cases that $X_0$ is not smoothable. 
 For example, in (\ref{eq:unsm}) we present the two cycles on the depicted $\hat{P}$ 
 satisfying the condition that $P_C$ is a chain,
\begin{align}
 \label{eq:unsm}
 \mbox{\begin{tikzpicture}[scale=0.2,baseline={(0.4,0.3)}]
  \fill[gray!40!white] (0,2) -- (3,1) -- (2,2) -- (1,3) -- (0,2) -- cycle;
  \draw[black] (0,2) -- (1,1) -- (2,0) -- (3,1) 
  -- (4,2) -- (3,3) -- (2,4) -- (1,3) -- (0,2) -- cycle;
  \draw[black] (1,1) -- (2,2) -- (3,3);
  \draw[black] (1,3) -- (2,2) -- (3,1);
  \draw[black] (0,2) -- (3,1);
  \filldraw[black] (0,2) circle (4pt);
  \filldraw[black] (1,1) circle (4pt);
  \filldraw[black] (1,3) circle (4pt);
  \filldraw[black] (2,2) circle (4pt);
  \filldraw[black] (3,1) circle (4pt);
  \filldraw[black] (3,3) circle (4pt);
  \filldraw[black] (4,2) circle (4pt);
  \draw (2,0-0.3) node[circle] {*};
  \draw (2,4-0.3) node[circle] {*};
\end{tikzpicture}}=
 \mbox{\begin{tikzpicture}[scale=0.2,baseline={(0.4,0.3)}]
  \fill[gray!40!white] (0,2) -- (1,1) -- (2,2) -- (1,3) -- (0,2) -- cycle;
  \draw[black] (0,2) -- (1,1) -- (2,0) -- (3,1) 
  -- (4,2) -- (3,3) -- (2,4) -- (1,3) -- (0,2) -- cycle;
  \draw[black] (1,1) -- (2,2) -- (3,3);
  \draw[black] (1,3) -- (2,2) -- (3,1);
  \draw[black] (0,2) -- (3,1);
  \filldraw[black] (0,2) circle (4pt);
  \filldraw[black] (1,1) circle (4pt);
  \filldraw[black] (1,3) circle (4pt);
  \filldraw[black] (2,2) circle (4pt);
  \filldraw[black] (3,1) circle (4pt);
  \filldraw[black] (3,3) circle (4pt);
  \filldraw[black] (4,2) circle (4pt);
  \draw (2,0-0.3) node[circle] {*};
  \draw (2,4-0.3) node[circle] {*};
\end{tikzpicture}}-
\mbox{\begin{tikzpicture}[scale=0.2,baseline={(0.4,0.3)}]
  \fill[gray!40!white] (0,2) -- (1,1) -- (1.5,1.5) -- (0,2) -- cycle;
  \fill[gray!80!white] (2,2) -- (3,1) -- (1.5,1.5) -- (2,2) -- cycle;
  \draw[black] (0,2) -- (1,1) -- (2,0) -- (3,1) 
  -- (4,2) -- (3,3) -- (2,4) -- (1,3) -- (0,2) -- cycle;
  \draw[black] (1,1) -- (2,2) -- (3,3);
  \draw[black] (1,3) -- (2,2) -- (3,1);
  \draw[black] (0,2) -- (3,1);
  \filldraw[black] (0,2) circle (4pt);
  \filldraw[black] (1,1) circle (4pt);
  \filldraw[black] (1,3) circle (4pt);
  \filldraw[black] (2,2) circle (4pt);
  \filldraw[black] (3,1) circle (4pt);
  \filldraw[black] (3,3) circle (4pt);
  \filldraw[black] (4,2) circle (4pt);
  \draw (2,0-0.3) node[circle] {*};
  \draw (2,4-0.3) node[circle] {*};
\end{tikzpicture}} \ \text{ and } \ \ 
\mbox{\begin{tikzpicture}[scale=0.2, baseline={(0.4,0.3)}]
  \fill[gray!40!white] (2,2) -- (3,1) -- (4,2) -- (3,3) -- (2,2) -- cycle;
  \draw[black] (0,2) -- (1,1) -- (2,0) -- (3,1) 
  -- (4,2) -- (3,3) -- (2,4) -- (1,3) -- (0,2) -- cycle;
  \draw[black] (1,1) -- (2,2) -- (3,3);
  \draw[black] (1,3) -- (2,2) -- (3,1);
  \draw[black] (0,2) -- (3,1);
  \filldraw[black] (0,2) circle (4pt);
  \filldraw[black] (1,1) circle (4pt);
  \filldraw[black] (1,3) circle (4pt);
  \filldraw[black] (2,2) circle (4pt);
  \filldraw[black] (3,1) circle (4pt);
  \filldraw[black] (3,3) circle (4pt);
  \filldraw[black] (4,2) circle (4pt);
  \draw (2,0-0.3) node[circle] {*};
  \draw (2,4-0.3) node[circle] {*};
\end{tikzpicture}} \ .
\end{align}
 The former cycle is a linear combination of remaining cycles
 as expressed by abuse of notation.
 However, the latter cycle is linearly independent to other cycles.
 Therefore, $X_0=(1^4)$ is not smoothable by Proposition~\ref{prop:smoothing}.
\end{example}


\section{Smoothing of CICY threefolds in Hibi toric varieties}
\subsection{Simply-connectedness}
\begin{proposition}[Corollary of Proposition~\ref{prop:small}]
 \label{prop:sc}
 Let $P$ be a pure poset, $X_0$ a CICY threefold in $\bP_{\Delta(P)}$,
 and $Y\rightarrow X_0$ a small resolution.
 Then $Y$ is simply-connected.
 If a smoothing $X \leadsto X_0$ 
 exists, $X$ is also simply-connected.
\end{proposition}
\begin{proof}
 By Proposition~\ref{prop:small},
 we have a small resolution $X_{\widehat{\Sigma}} \rightarrow X_\Sigma=\bP_{\Delta(P)}$.
 Since $X_{\widehat{\Sigma}}$ is a compact toric variety,
 it is simply-connected (see for example \cite[Theorem 9.1]{MR495499}).
 Let $\Sigma^{(1)}$ denote the subfan 
 of $\Sigma$ consisting of cones of dimension less than or equal to one.
 Note $\widehat{\Sigma}^{(1)} = \Sigma^{(1)}$.
 The quasi-projective toric variety $X_{\Sigma^{(1)}}$
 is simply-connected as well.
 In fact, the difference between 
 $X_{\widehat{\Sigma}}$ and $X_{\Sigma^{(1)}}$
 in subsets of real codimension four
 does not effect fundamental groups.
 By the Lefschetz theorem for non-singular quasi-projective manifolds 
 \cite{MR932724,MR820315}, 
 a complete intersection
 $X_{\Sigma^{(1)}}\cap X_0 
 = X_{\widehat{\Sigma}^{(1)}}\cap Y$ is also simply-connected.
 Similarly as above,
 the difference between $X_{\widehat{\Sigma}^{(1)}}\cap Y$ and $Y$
 does not effect fundamental groups.
 Therefore $Y$ is also simply-connected.
 The latter statement follows from the fact that 
 a conifold transition does not change fundamental groups.
\end{proof}

 Let $X$ be a smoothing of a CICY threefold in Hibi toric variety.
 We do not know whether homology groups of $X$ can have torsion or not.
 Suppose that $X$ has torsion-free homology and $h^{1,1}(X)=1$.
 In this case, by Wall's theorem \cite[Theorem 5]{MR0215313}, 
 the diffeomorphism class
 of $X$ is determined only by the three topological invariants,
 \begin{align}
  \label{eq:three}
  \deg X, \ c_2(X)\cdot H, \  \text{ and } \ \chi(X), 
 \end{align} 
 where 
 $H$ is the hyperplane class,
 $c_2(X)$ is the second Chern class
 and $\chi(X)=2\left(h^{1,1}(X)-h^{2,1}(X)\right)$ is the topological Euler number
 of $X$. 
 We summarize the calculation of these topological invariants
 in the next subsection.

\subsection{Topological invariants}
\label{subsec:top}
For a conifold transition 
$X \leadsto X_0 \leftarrow Y$,
Hodge numbers satisfy 
\begin{align}
 \label{eq:h11x}
 h^{1,1}(X) &= h^{1,1}(Y) - \mathrm{rk},\\
 \label{eq:h21x}
 h^{1,2}(X) &= h^{1,2}(Y) + \mathrm{dp} - \mathrm{rk},
\end{align}
where 
$\mathrm{rk}=\mathrm{rk}(X_0)$
is the dimension of
linear subspace of
$H_2(Y,\bC)$
spanned by classes of
exceptional lines $[L_i]$ for $i=1,\dots \mathrm{dp}$, i.e.,
\begin{align}
 \mathrm{rk}(X_0)= \rank\left( \rho_C \, \middle| \, C\in \Lambda_4(\hat{P})  \right).
\end{align}
and 
$\mathrm{dp}=\mathrm{dp}(X_0)$
is the number of 
nodes on $X_0$, which we compute by (\ref{eq:nodes}).
In particular, from
$h^{1,1}(Y)=\abs{E}-\abs{P} = b_1(\hat{P})+1$,
we have  $h^{1,1}(X)=1$ if and only if
all minimal cycles
on $\hat{P}$ are generated by cycles in 
$\Lambda_4(\hat{P})$. 

Assume $h^{1,1}(X)=1$. 
From (\ref{eq:degree}) and the invariance by a flat deformation, it holds
\begin{align}
 \label{eq:deg}
 \deg X = \deg X_0 =
 c_{J(P)} \prod_{j=1}^r d_j.
\end{align}
Since also the invariance
$\chi(X,\scO_X(1)) =\chi(X_0,\scO_{X_0}(1))$ and a standard 
cohomology calculation 
for complete intersection varieties in $\bP_{\Delta(P)}$,
we obtain
\begin{align}
 \chi(X,\scO_{X}(1)) & = \dim H^0(X_0,\scO_{X_0}(1)) \\
 & = \dim H^0(\bP_{\Delta(P)},\scO_{\bP_{\Delta(P)}}(1)) - r_1\\
 & = \abs{J(P)}-r_1,
\end{align}
where $r_1= \# \left\{ j\mid d_j=1 \right\}$.
Therefore the Hirzebruch--Riemann--Roch theorem gives
\begin{align}
 c_2(X)\cdot H & = 12 \chi(X,\scO_{X}(1)) -2 \deg X\\
 \label{eq:c2}
 & = 12 \left( \abs{J(P)} -r_1 \right) -2c_{J(P)} \prod_{j=1}^r d_j.
\end{align}

\begin{example}
 For $P= \scP_1$, $h^{1,1}(X)=1$ holds since all minimal cycles are in
 $\Lambda_4(\hat{\scP_1})$.
 Hence $\mathrm{rk}(X_0) = 5$ by (\ref{eq:h11x}).
 From $h^{1,2}(Y)=33$ and $\mathrm{dp}(X_0)=12$,
 it holds $h^{1,2}(X)=40$ and $\chi(X) = 2(h^{1,1}(X)-h^{1,2}(X))=-78$
 by (\ref{eq:h21x}).
 We also obtain $\deg(X)=48$ and $c_2(X)\cdot H = 12(18-3)-2\cdot 48 = 84$ 
 by (\ref{eq:deg}) and (\ref{eq:c2}), respectively.
\end{example}


\subsection{Examples}
\label{ss:new}
Let $P$ be a pure poset and $X_0$ a CICY threefold in $\bP_{\Delta(P)}$.
We write $X=X_P$ for a smoothing $X\leadsto X_0$ if it exists.
In spite of the large number of smoothable
CICY threefolds in Hibi toric varieties,
there are few examples of $X_P$ with $h^{1,1}(X_P)=1$.
From \cite[Proposition 3.1]{MR3688804},
we have twelve such threefolds
as complete intersections in minuscule Schubert varieties
up to deformation equivalence;
five in projective spaces,
five in other Grassmannians $G(k,n)$,
one in an orthogonal Grassmannian $OG(5,10)$,
and one in a singular Schubert variety of the Cayley plane $\bO \bP^2$.
The latter two threefolds are also regarded as
complete intersections of some homogeneous vector bundles
on Grassmannians, i.e., No.\,4 and No.\,7 in \cite[Table 1]{2016arXiv160707821I}, respectively.
Apart from these twelve examples,
we introduce six more examples in Table~\ref{tab:CY3}. The topological invariants
are computed by formulas in the previous subsection.
\begin{table}[H]
 \caption{Examples of $X$ of Picard number one.}
 \centering 
 \vspace{4pt}
 {\begin{tabular}{@{}c|cccccc@{}}
 \toprule
 posets $\scP_i$ \rule[0mm]{0mm}{4mm} & $\scP_1$ & $\scP_2$ & $\scP_3$ & $\scP_4$ & $\scP_5$& $\scP_6$\\
 Hasse diagrams \rule[-4mm]{0mm}{10mm} &
 \begin{tikzpicture}[scale=0.4, baseline={(0,0.1)}]
  \draw[black] (0,0) -- (0,1) -- (1,0) -- (2,1) -- (2,0) -- (1,1) -- (0,0) -- cycle;
  \filldraw[black] (0,0) circle (2pt);
  \filldraw[black] (0,1) circle (2pt);
  \filldraw[black] (1,0) circle (2pt);
  \filldraw[black] (1,1) circle (2pt);
  \filldraw[black] (2,0) circle (2pt);
  \filldraw[black] (2,1) circle (2pt);
 \end{tikzpicture}
 &
 \begin{tikzpicture}[scale=0.2, baseline={(0,0.2)}]
  \draw[black] (0,3) -- (1,2) -- (0,1) -- (1,0) 
  -- (2,1) -- (1,2) -- (2,3) -- (3,2) -- (2,1) -- (3,0);
  \filldraw[black] (0,1) circle (4pt);
  \filldraw[black] (0,3) circle (4pt);
  \filldraw[black] (1,0) circle (4pt);
  \filldraw[black] (1,2) circle (4pt);
  \filldraw[black] (2,1) circle (4pt);
  \filldraw[black] (2,3) circle (4pt);
  \filldraw[black] (3,0) circle (4pt);
  \filldraw[black] (3,2) circle (4pt);
 \end{tikzpicture}
 &
 \begin{tikzpicture}[scale=0.2, baseline={(0,0.1)}]
  \draw[black] (0,2) -- (1,1) -- (2,2) -- (3,1) -- (4,2);
  \draw[black] (0,0) -- (1,1) -- (2,0) -- (3,1);
  \filldraw[black] (0,0) circle (4pt);
  \filldraw[black] (0,2) circle (4pt);
  \filldraw[black] (1,1) circle (4pt);
  \filldraw[black] (2,0) circle (4pt);
  \filldraw[black] (2,2) circle (4pt);
  \filldraw[black] (3,1) circle (4pt);
  \filldraw[black] (4,2) circle (4pt);
 \end{tikzpicture}
 &
 \begin{tikzpicture}[scale=0.4, baseline={(0,0.1)}]
  \draw[black] (0,0) -- (0,1) -- (1,0) -- (1,1) -- (2,0) -- (2,1);
  \filldraw[black] (0,0) circle (2pt);
  \filldraw[black] (0,1) circle (2pt);
  \filldraw[black] (1,0) circle (2pt);
  \filldraw[black] (1,1) circle (2pt);
  \filldraw[black] (2,0) circle (2pt);
  \filldraw[black] (2,1) circle (2pt);
 \end{tikzpicture}
 &
 \begin{tikzpicture}[scale=0.2, baseline={(0,0.08)}]
  \draw[black] (0,2) -- (1,0) -- (2,2) -- (3,0) -- (4,2);
  \filldraw[black] (0,2) circle (4pt);
  \filldraw[black] (1,0) circle (4pt);
  \filldraw[black] (2,2) circle (4pt);
  \filldraw[black] (3,0) circle (4pt);
  \filldraw[black] (4,2) circle (4pt);
 \end{tikzpicture}
 &
 \begin{tikzpicture}[scale=0.266, baseline={(0,0.04)}]
  \draw[black] (0,0) -- (0,1) -- (1,0) -- (1,1);
  \filldraw[black] (0,0) circle (3pt);
  \filldraw[black] (0,1) circle (3pt);
  \filldraw[black] (1,0) circle (3pt);
  \filldraw[black] (1,1) circle (3pt);
 \end{tikzpicture}\,$\oplus$\!
 \begin{tikzpicture}[scale=0.266, baseline={(0,0.04)}]
  \draw[black] (0,0) -- (0,1) -- (1,0) -- (1,1);
  \filldraw[black] (0,0) circle (3pt);
  \filldraw[black] (0,1) circle (3pt);
  \filldraw[black] (1,0) circle (3pt);
  \filldraw[black] (1,1) circle (3pt);
 \end{tikzpicture}\\
 degrees \rule[-2mm]{0mm}{0mm}& $(1^3)$ & $(1^5)$ & $(1^4)$ & $(1^3)$ & $(1,2)$ & $(1^5)$\\
 \hline
 $\deg X_{\scP_i}$\rule[-2mm]{0mm}{6mm}&$48$& $29$ &$42$ &$61$ &$32$ & $25$\\
 $c_2(X_{\scP_i})\cdot H$\rule[-2mm]{0mm}{6mm}&$84$&$74$&$84$&$94$ & $80$ &$70$\\
 $\chi(X_{\scP_i})$\rule[-2mm]{0mm}{6mm}&$-78$&$-100$&$-96$ & $-86$& $-116$ &$-100$\\
 \bottomrule
\end{tabular}}
\label{tab:CY3}
\end{table}

Note that $X_{\scP_4}$, $X_{\scP_5}$ and $X_{\scP_6}$
are deformation equivalent to 
complete intersections of some homogeneous vector bundles on Grassmannians, i.e.,
No.\,23, No.\,20 and No.\,5 in \cite[Table 1]{2016arXiv160707821I}, respectively
(from a private communication with Daisuke Inoue and Atsushi Ito for $X_{\scP_4}$).
Furthermore, $X_{\scP_5}$ and $X_{\scP_6}$ are also 
regarded as
a complete intersection $(1^2,2)$ in a Lagrangian Grassmannian $LG(3,6)$,
and a complete intersection of two Grassmannians $G(2,5)$ in $\bP^9$, respectively.

\begin{remark}
 \label{rem:pf}
Let us discuss mirror symmetry for examples in Table~\ref{tab:CY3}.
For each conifold transition $X\leadsto X_0 \leftarrow Y$,
we expect another conifold transition $Y^* \leadsto Y_0^* \leftarrow X^*$
in the mirror side, if $X$ and $Y$ have torsion-free homology.
We have a Batyrev--Borisov mirror $Y^*$ for $Y \subset \bP_{\widehat{\Sigma}}$,
and the degeneration $Y^* \leadsto Y_0^*$
corresponding to the small resolution $Y\rightarrow X_0$ 
based on the argument by
\cite{MR2112571}.
However, we do not know whether $Y_0^*$ has the same number of nodes as $X_0$
and admits a small resolution $X^* \rightarrow Y_0^*$ or not.
In spite of that, 
periods and the Picard--Fuchs operator vanishing the periods
for the conjectural mirror family
are computable in advance.
The resulting operators in the case of
$\scP_2$, $\scP_3$, $\scP_4$, $\scP_5$ and $\scP_6$
coincide with already known operators,
\#195, \#28, \#124, \#42, and \#101 in \cite{CYequations,2005math......7430A}, respectively.

The operator for $\scP_1$ seems unknown, thus we write it here.
A formula for the fundamental period $\omega_0(z) = \sum_{m=0}^\infty A_m z^m$
is given in \cite[Eq.\,(5,2)]{MR3688804}, where
\begin{align}
 A_m= \sum_{s,t,u,v,w}
 \!\!\!\begin{pmatrix}s\\u\end{pmatrix}
 \!\!\begin{pmatrix}v\\s\end{pmatrix}
 \!\!\begin{pmatrix}t\\s\end{pmatrix}
 \!\!\begin{pmatrix}t\\v\end{pmatrix}
 \!\!\begin{pmatrix}w\\t\end{pmatrix}
 \!\!\begin{pmatrix}m\\t\end{pmatrix}
 \!\!\begin{pmatrix}m\\w\end{pmatrix}
 \!\!\begin{pmatrix}v-t+w\\u-s+v\end{pmatrix}
 \!\!\begin{pmatrix}m\\v-t+w\end{pmatrix}
\end{align}
is read from the (slightly modified) dual graph of $\hat{\scP}_1$,
by associating binomial coefficients with oriented edges and
linear relations with pairs of dashed edges:
\begin{align}
 \begin{tikzpicture}[scale=1, baseline={(0,1.4)}]
 \draw[gray, densely dotted] (0,1) -- (0,2) -- (1,1) -- (2,2) -- (2,1) -- (1,2) -- (0,1) -- cycle;
 \draw[gray, densely dotted] (0,1) -- (1,0.15);
 \draw[gray, densely dotted] (1,1) -- (1,0.15);
 \draw[gray, densely dotted] (2,1) -- (1,0.15);
 \draw[gray, densely dotted] (1,2) -- (1,2.85);
 \draw[gray, densely dotted] (2,2) -- (1,2.85);
 \draw[gray, densely dotted] (0,2) -- (1,2.85);
 \filldraw[gray] (0,1) circle (0.8pt);
 \filldraw[gray] (0,2) circle (0.8pt);
 \filldraw[gray] (1,1) circle (0.8pt);
 \filldraw[gray] (2,2) circle (0.8pt);
 \filldraw[gray] (2,1) circle (0.8pt);
 \filldraw[gray] (1,2) circle (0.8pt);
 \draw[gray] (1,0.15 - 0.06) node[circle] {*}; 
 \draw[gray] (1,2.85 - 0.06) node[circle] {*}; 
 \draw (0.6,2.05) node[circle] {\small{$s$}}; 
 \draw (1.4,2.05) node[circle] {\small{$t$}}; 
 \draw (0.2,1.5) node[circle] {\small{$u$}}; 
 \draw (1,1.5) node[circle] {\small{$v$}}; 
 \draw (1.8,1.5) node[circle] {\small{$w$}}; 
 \draw (2.6,1.5) node[circle] {\small{$m$}}; 
 \draw (0.6,0.95) node[circle] {\small{$a$}};
 \draw (1.4,0.95) node[circle] {\small{$b$}}; 
 \draw[->] (0.3,1.65)--(0.55,1.9);
 \draw[->] (0.65,1.9)--(0.9,1.65);
 \draw[->] (1.1,1.65)--(1.35,1.9);
 \draw[->] (1.45,1.9)--(1.7,1.65);
 \draw[dashed] (0.3,1.35)--(0.55,1.1);
 \draw[dashed] (0.65,1.1)--(0.9,1.35);
 \draw[dashed] (1.1,1.35)--(1.35,1.1);
 \draw[dashed] (1.45,1.1)--(1.7,1.35);
 \draw[->] (1.95,1.5)--(2.4,1.5);
 \draw[->] (0.7,2.15) to [out=30,in=150] (1.3,2.15);
 \draw[->] (0.7,0.8) to [out=-30,in=-150] (1.3,0.8);
 \draw[->] (1.5,2.15) to [out=50,in=120] (2.45,1.65);
 \draw[->] (1.5,0.8) to [out=-50,in=-120] (2.45,1.35);
 \draw[gray!60!white] (-0.6,1.5) node[circle] {\small{$0$}}; 
 \draw[gray!60!white, ->] (-0.4,1.5) -- (0.05,1.5);
 \draw[gray!60!white, ->] (-0.55, 1.65) to [out=60,in=130] (0.5,2.15);
 \draw[gray!60!white, ->] (-0.55,1.35) to [out=-60,in=-130] (0.5,0.8);
\end{tikzpicture}
\qquad
\begin{cases}
 a = u-s+v,\\
 b = v-t+w.
\end{cases}
\end{align}
With the aid of numerical method, we obtain the following Picard--Fuchs operator
for a conjectural mirror family for $X_{\scP_1}$,
\begin{align}
 \label{eq:op}
\begin{split}
\scD =
\theta^4 &- 2z (33 \theta^4 + 58 \theta^3 + 48 \theta^2+ 19 \theta + 3)\\ 
&+ 4z^2 (174 \theta^4+ 448 \theta^3+ 527 \theta^2+ 314 \theta +75)\\
&- 8z^3 ( 332 \theta^4+ 1096 \theta^3+ 1507 \theta^2+ 953 \theta+228)\\
&+ 96 z^4(\theta + 1)^2 (6 \theta + 5) (6 \theta+7),
\end{split}
\end{align}
where $\theta = z \partial_z$ and $\scD \omega_0(z)= 0$.
We observe that the operator
generates integral BPS numbers 
for genus $0$ and genus $1$ with small degrees,
by standard methods for the computation
\cite{MR1115626, MR1240687}.
\end{remark}


\section*{Acknowledgments}
 This article is an extended version of the author's talk at
 Algebraic and Geometric Combinatorics on Lattice Polytopes 2018.
 He would like to thank the organizers for the kind invitation.
 He is also grateful to 
 Tomoyuki Hisamoto, Atsushi Ito and Taro Sano
 for helpful discussions.



\begin{thebibliography}{999}
  \bibitem{2005math......7430A}
   G. Almkvist, C. van Enckevort, D. van Straten, and W. Zudilin, 
   {\it Tables of Calabi--Yau equations\/}, [arXiv:math.AG/0507430].
  \bibitem{MR2112571}
   V. V. Batyrev, {\it Toric degenerations of Fano varieties and constructing mirror
   manifolds\/}, in The Fano Conference (Univ. Torino, Turin, 2004), pp.~109--122.
  \bibitem{MR1269718}
   V. V. Batyrev, {\it Dual polyhedra and mirror symmetry for Calabi--Yau 
   hypersurfaces in toric varieties\/}, J. Algebraic Geom. 3, 493--535 (1994).
  \bibitem{MR1463173}
   V. V. Batyrev and L. A. Borisov, {\it On Calabi-Yau complete intersections in
   toric varieties\/}, in Higher-dimensional complex varieties (Trento, 1994) 
   (de Gruyter, Berlin, 1996), pp.~39--65.
  \bibitem{MR1619529}
   V. V. Batyrev, I. Ciocan-Fontanine, B. Kim, and D. van Straten,
   {\it Conifold transitions and mirror symmetry for Calabi-Yau complete intersections 
   in Grassmannians\/}, Nuclear Phys. B \textbf{514}, 640--666 (1998).
  \bibitem{MR1240687}
   M. Bershadsky, S. Cecotti, H. Ooguri, and C. Vafa, {\it Holomorphic anomalies
   in topological field theories\/}, Nuclear Phys. B \textbf{405}, 279--304 (1993).
  \bibitem{MR2801412}
   V. Batyrev and M. Kreuzer, {\it Constructing new Calabi-Yau 3-folds and their
   mirrors via conifold transitions\/}, Adv. Theor. Math. Phys. \textbf{14}, 879--898 (2010).
  \bibitem{1993alg.geom.10001B}
   L. Borisov, {\it Towards the Mirror Symmetry for Calabi-Yau 
    Complete intersections in Gorenstein Toric Fano Varieties\/}, 
    [arxiv:alg-geom/9310001].
  \bibitem{MR1115626}
   P. Candelas, X. C. de la Ossa, P. S. Green, and L. Parkes, {\it 
   A pair of Calabi--Yau manifolds as an exactly soluble superconformal theory\/},
   Nuclear Phys. B \textbf{359}, 21--74 (1991)
  \bibitem{MR495499}
   V. I. Danilov, {\it The geometry of toric varieties\/}, Uspekhi Mat. Nauk \textbf{33}, 
   85--134, 247 (1978).
  \bibitem{MR932724}
   M. Goresky and R. MacPherson, {\it Stratified Morse theory\/}, Vol.~14, Ergebnisse
   der Mathematik und ihrer Grenzgebiete (3) [Results in Mathematics and
   Related Areas (3)] (Springer-Verlag, Berlin, 1988), pp.~xiv+272.
  \bibitem{MR1086756}
   A. He and P. Candelas, {\it On the number of complete intersection Calabi--Yau
   manifolds\/}, Comm. Math. Phys. \textbf{135}, 193--199 (1990)
  \bibitem{MR2770546}
   T. Hibi and A. Higashitani, {\it Smooth Fano polytopes arising from finite 
   partially ordered sets\/}, Discrete Comput. Geom. \textbf{45}, 449--461 (2011).
  \bibitem{MR820315}
   H. A. Hamm and L. D. Trang, {\it Lefschetz theorems on quasiprojective varieties\/},
   Bull. Soc. Math. France 113, 123--142 (1985).
  \bibitem{MR790025}
   T. Hibi and K. Watanabe, {\it Study of three-dimensional algebras with 
   straightening laws which are Gorenstein domains. I}, Hiroshima Math. J. \textbf{15}, 27--54
   (1985).
  \bibitem{2016arXiv160707821I}
   D. Inoue, A. Ito, and M. Miura, {\it Complete intersection Calabi--Yau manifolds
   with respect to homogeneous vector bundles on Grassmannians}, [arXiv:math.AG/1607.07821].
  \bibitem{MR3688804}  
   M. Miura, {\it Minuscule Schubert varieties and mirror symmetry\/}, 
   SIGMA Symmetry Integrability Geom. Methods Appl. \textbf{13}, Paper No.~067, 25 (2017).
  \bibitem{MR1673108}
   D. R. Morrison, {\it Through the looking glass\/}, 
   in Mirror symmetry, I\!I\!I (Montreal, PQ, 1995), 
   Vol.~10, AMS/IP Stud. Adv. Math. (Amer. Math. Soc.,
   Providence, RI, 1999), pp.~263--277.
  \bibitem{MR1923221}
   Y. Namikawa, {\it Stratified local moduli of Calabi-Yau threefolds\/}, Topology
   \textbf{41}, 1219--1237 (2002).
  \bibitem{MR909231}
   M. Reid, {\it The moduli space of $3$-folds with $K = 0$ may nevertheless be 
   irreducible\/}, Math. Ann. \textbf{278}, 329--334 (1987).
  \bibitem{CYequations}
   D. van Straten, Calabi--Yau Database, Version 2.0 
   (\url{www2.mathematik.uni-mainz.de/CYequations/db/}), 
   Version 3.0 (\url{cydb.mathematik.uni-mainz.de}).
  \bibitem{MR1382045}
   D. G. Wagner, {\it Singularities of toric varieties associated 
   with finite distributive lattices\/}, J. Algebraic Combin. \textbf{5}, 149--165 (1996).
  \bibitem{MR0215313}
   C. T. C. Wall, {\it Classification problems in differential topology. V. 
    On certain 6-manifolds\/}, Invent. Math. 1 (1966), 355--374; corrigendum, ibid \textbf{2}, 
    306 (1966).
\end{thebibliography}
\end{document}